\documentclass[reqno]{amsart}
\usepackage[centering,textwidth=6.5in,textheight=9in, marginparwidth=.8in]{geometry}
\usepackage{amsmath,amssymb,amscd,mathtools}
\usepackage{yhmath}
\usepackage{setspace}

\usepackage{tikz}
\usetikzlibrary{math,angles,quotes,intersections,calc,through}

\usepackage{hyperref}

\newtheorem{thm}{Theorem}[section]
\newtheorem{lem}[thm]{Lemma}
\newtheorem{pro}[thm]{Proposition}

\theoremstyle{definition}
\newtheorem{definition}[thm]{Definition}
\newtheorem{remark}[thm]{Remark}

\numberwithin{equation}{section}

\def\bR{\mathbb{R}}
\def\bT{\mathbb{T}}
\def\bZ{\mathbb{Z}}

\def\cC{\mathcal{C}}

\def\cO{\mathcal{O}}

\def\cS{\mathcal{S}}

\def\hF{\widehat{F}}
\def\hM{\widehat{M}}

\def\pa{\partial}

\def\ds{\displaystyle}

\begin{document}

\title{Hyperbolicity of asymmetric lemon billiards}

\author[X. Jin]{Xin Jin}
\address{Math Department,
Boston College,
Chestnut Hill, MA 02467.}
\email{xin.jin@bc.edu}
\thanks{X. J. is supported in part by  the NSF Grant DMS-1854232.}

\author[P. Zhang]{Pengfei Zhang}
\address{Department of Mathematics,
University of Oklahoma,
Norman, OK 73069.}
\email{pengfei.zhang@ou.edu}



\begin{abstract}
Asymmetric lemon billiards was introduced in \cite{CMZZ},
where the billiard table $Q(r,b,R)$ is the intersection of two round disks with radii $r\le R$, respectively,
and $b$ measures the distance between the two centers. 
It is conjectured \cite{BZZ} that the asymmetric lemon billiards is hyperbolic when the arc $\Gamma_r$
is a major arc and $R$ is large. 
In this paper we prove this conjecture for sufficiently large $R$.
\end{abstract}

\maketitle

\section{Introduction}

{\it Dynamical billiards} is a special class of  dynamical systems, 
in which a point particle alternates between moving freely inside a bounded domain $Q$ and
elastic reflections upon hitting the boundary  $\Gamma=\pa Q$.
 The domain $Q$ is called the billiard table.
The dynamical properties of billiards are determined completely by the 
geometric shape of the billiard table.
For example, Jacobi proved the dynamical billiards on an elliptic table is completely integrable.

The study of {\it chaotic billiards} was
pioneered by Sina\v{\i}. In his seminal paper \cite{Si70},
Sina\v{\i} discovered the {\it dispersing mechanism} and proved the hyperbolicity and ergodicity 
of dispersing billiards. The  dispersing mechanism
states that any parallel (divergent) beam of trajectories becomes (more) divergent after 
reflection from a dispersing boundary. See Fig.~\ref{dis-foc}.
Bunimovich \cite{Bu74a} constructed a family of chaotic billiard systems with a mixture 
of dispersing and focusing components.
In \cite{Bu74b} he  constructed a family of chaotic billiard systems with focusing and neutral components only,
and formulated the first version of {\it defocusing mechanism} for chaotic billiards.
The defocusing mechanism have been greatly extended by Wojtkowski \cite{Woj86},
Markarian \cite{Ma88}, Donnay \cite{Do91} and Bunimovich \cite{Bu92}.
Generally speaking,  defocusing mechanism applies 
if  all free paths are long enough such that parallel beams of trajectories, 
becoming convergent after reflection from a focusing boundary, 
pass the convergent points and become divergent.

\begin{figure}[h]
\tikzmath{
\a=150;
\b=170; 
\c=160; 
}
\begin{tikzpicture}
\coordinate (O) at (4,0);
\coordinate (O1) at (1,0);
\filldraw (O) circle(0.03) node[right]{$O$};
\coordinate (A) at ({4+2*cos(\a)},{2*sin(\a)});
\draw (A) arc (\a: 360-\a : 2);
\coordinate (B1) at ({4+2*cos(\b)},{2*sin(\b)});
\draw[dashed, gray] (O) -- ({4+ 2.5*cos(\b)},{2.5*sin(\b)});
\coordinate (B2) at ({4+2*cos(180)},{2*sin(180)});
\draw[dashed, gray] (O) -- ({4+ 2.5*cos(180)},{2.5*sin(180)});
\coordinate (B3) at ({4+2*cos(-\b)},{2*sin(-\b)});
\draw[dashed, gray] (O) -- ({4+ 2.5*cos(-\b)},{2.5*sin(-\b)});
\draw[blue] (B1) -- +({2*cos(\c)},{2*sin(\c)}) (B2) -- +({2*cos(\c)},{2*sin(\c)}) (B3) -- +({2*cos(\c)},{2*sin(\c)});
\draw[red, ->] (B1) -- +({1.1*cos(180)},{1.1*sin(180)});
\draw[red, ->] (B2) -- +({1*cos(-\c)},{1*sin(-\c)});
\draw[red, ->] (B3) -- +({1.2*cos(-140)},{1.2*sin(-140)});
\end{tikzpicture}
\hspace{1in}
\tikzmath{
\a=30;
\b=10; 
\c=160; 
}
\begin{tikzpicture}
\coordinate (O) at (0,0);
\coordinate (O1) at (1,0);
\filldraw (O) circle(0.03) node[left]{$O$};
\coordinate (A) at ({2*cos(\a)},{2*sin(\a)});
\draw (A) arc (\a: -\a : 2);
\coordinate (B1) at ({2*cos(\b)},{2*sin(\b)});
\draw[dashed, gray] (O) -- ({ 2*cos(\b)},{2*sin(\b)});
\coordinate (B2) at ({2*cos(0)},{2*sin(0)});
\draw[dashed, gray] (O) -- ({ 2*cos(0)},{2*sin(0)});
\coordinate (B3) at ({2*cos(-\b)},{2*sin(-\b)});
\draw[dashed, gray] (O) -- ({ 2*cos(-\b)},{2*sin(-\b)});
\draw[blue] (B1) -- +({2*cos(\c)},{2*sin(\c)}) (B2) -- +({2*cos(\c)},{2*sin(\c)}) (B3) -- +({2*cos(\c)},{2*sin(\c)});
\draw[red, ->] (B1) -- +({2*cos(-140)},{2*sin(-140)});
\draw[red, ->] (B2) -- +({1.8*cos(-\c)},{1.8*sin(-\c)});
\draw[red, ->] (B3) -- +({1.7*cos(180)},{1.7*sin(180)});
\end{tikzpicture}
\caption{Reflections of a parallel beam on a dispersing boundary (left) 
and  on a focusing boundary (right), respectively.}
\label{dis-foc}
\end{figure}
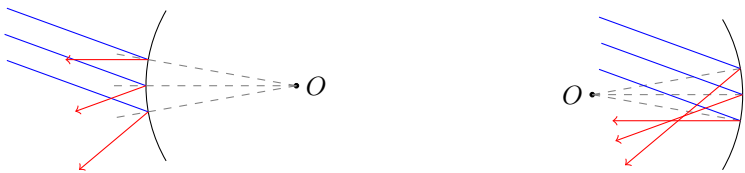

In \cite{HeTo} Heller and Tomsovic studied some lemon-shaped billiard systems,
where the billiard table $Q(b)$
is the intersection of two unit disks whose centers are separated by $b$ units, $0< b < 2$.
Numerical studies have been done extensively for  the lemon billiards 
in relation to the problems of quantum chaos (see \cite{MHA,RR99}). 
Recently, the existence of elliptic islands for lemon billiards has been proved in \cite{OP05}.
In \cite{CMZZ} we considered the asymmetric lemon-shaped billiards,
where the billiard table $Q(r,b,R)$ 
is the intersection of two round disks of radii $r\le R$, respectively, 
whose centers $O_r$ and $O_R$ are separated by $b$ units, $R-r < b< R+r$. 
See Fig.~\ref{asl} for an example of the asymmetric lemon billiard table.
One can assume $r=1$ without losing any generality. 
We will keep using $r$ to emphasize the role of the radius $r$, although $r=1$.

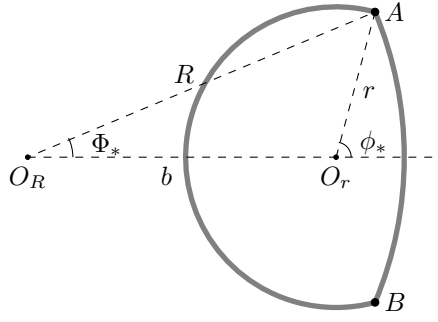
\begin{figure}[h]
\tikzmath{
\a=75;
\b=22.73; 
}
\begin{tikzpicture}
\coordinate (OR) at (0,0);
\filldraw (OR) node[below] {$O_R$} circle(0.03);
\coordinate (Or) at (4.1,0);
\filldraw (Or) node[below] {$O_r$} circle(0.03);
\coordinate (A) at ({4.1+2*cos(\a)},{2*sin(\a)});
\draw[gray, line width=2] (A) arc (\b: -\b :5);
\draw[gray, line width=2] (A) arc (\a:(360 - \a):2);
\filldraw (A) circle(.05)  node[right] {$A$};
\coordinate (B) at ({4.1+2*cos(-\a)},{2*sin(-\a)});
\filldraw (B) circle(.05)  node[right] {$B$};
\draw[dashed] (OR) -- (Or) node[pos=0.45,below]{$b$} -- (5.5,0); 
\draw[dashed] (Or) -- (A)  node[pos=0.45,right]{$r$} -- (OR)  node[pos=.55,above]{$R$};
\draw (0.6, 0) arc (0:25:0.6);
\node (Phi) at (1.05,.16){$\Phi_{\ast}$};
\draw (4.3, 0) arc (0:80:0.2);
\node (phi) at (4.6,.2){$\phi_{\ast}$};
\end{tikzpicture}
\caption{An asymmetric lemon table $Q(r,b,R)$.}
\label{asl}
\end{figure}

Given an asymmetric lemon table $Q(r,b,R)$, let $A,B$ be the two corners
where the two circular arcs $\Gamma_r$ and $\Gamma_R$ meet at, 
$\phi_{\ast}$ and $\Phi_{\ast}$ be the position angle of the point $A$ with respect to $O_r$
and $O_R$, respectively. 
Then the three parameters $r, b, R$ are related in the following way:
\begin{align}
R\sin\Phi_{\ast}= r\sin\phi_{\ast}, \quad
R\cos\Phi_{\ast}=b+ r\cos\phi_{\ast}. 
\label{rbR}
\end{align}
To make the corners of the table $Q(r,b,R)$ fixed at the given points $A$ and $B$ on $\pa D_r$,
we have 
\begin{align}
b= R\cos\Phi_{\ast} - r\cos\phi_{\ast}
=(R^2 - r^2\sin^2\phi_{\ast})^{1/2}  - r\cos\phi_{\ast}.
\label{distance}
\end{align}
Given $\phi_\ast \in (0, \frac{\pi}{2})$, let 
$Q(\phi_{\ast},R):=Q(1,b(\phi_{\ast}, R),R)$
be the family of asymmetric lemon tables with corners fixed at the two points $A$  and $B$,
where $b=b(\phi_{\ast}, R)$ is given by \eqref{distance}.

Note that when $R\to\infty$, the table $Q(\phi_{\ast},\infty)$ turns out to be
a flower table with one petal constructed by Bunimovich \cite{Bu74b, Bu79}.
It is proved that the one-petal billiards $Q(\phi_{\ast},\infty)$ is hyperbolic and ergodic.
The asymmetric lemon billiards $Q(\phi_{\ast},R)$, especially when $R\gg r$, 
may be viewed as a small {\it geometric deformation} of  $Q(\phi_{\ast},\infty)$.
However, a geometric deformation of the configuration space of a Hamiltonian system, 
no matter how small it is,
leads to a global change of the billiard map on the phase space. 
It has been observed  numerically in \cite{CMZZ} that there is an infinite strip
in the parameter space $\{(b,R): 1< b<R\}$ such that the asymmetric lemon billiards 
in that strip is ergodic. 
 In \cite{BZZ} we conjectured that if $\Gamma_r$ is a major arc
 (in the sense that the arc-length $|\Gamma_r|> \pi$), 
 then the asymmetric lemon billiards $Q(\phi_{\ast},R)$ 
is hyperbolic for large $R$. In \cite{BZZ} 
we have proved the hyperbolicity under the assumption that $\phi_{\ast}\in (0, \pi/6)$. 
We can remove this assumption now:

\begin{thm}\label{mainthm}
Let  $\phi_{\ast}\in(0,\pi/2)$, and  $Q(\phi_{\ast},R)$ be an asymmetric lemon table. 
Then for any 
$R\ge R(\phi_{\ast}):=\max\big\{ \frac{14.6r}{\min\{\phi_{\ast}, \frac{\pi}{2}-\phi_{\ast}\}\cdot \sin\phi_{\ast}}, \frac{147 r}{\sin^2\phi_{\ast}}, 1773.7 r\big\}$,
the asymmetric lemon billiards $Q(\phi_{\ast}, R)$ is hyperbolic.
\end{thm}
For example,  when $\phi_{\ast}\in [\frac{\pi}{6}, \frac{\pi}{2} - 0.0083]$, 
the asymmetric lemon billiards $Q(\phi_{\ast}, R)$ is hyperbolic for $R \ge 1773.7 r$.

The  strategy of our proof of Theorem \ref{mainthm} is 
to construct a mensurable cone field on the phase space
that is invariant and eventually strictly invariant under the billiard map.
Then the hyperbolicity of asymmetric lemon billiards follows from a classical result of
Wojtkowski \cite{Woj85}.

\begin{remark}
Note that the lower bound $ R(\phi_{\ast}) \to \infty$ when $\phi_{\ast} \to \frac{\pi}{2}$.
This is compatible with the fact that  the  table  $Q( \frac{\pi}{2}, \infty)$ 
is a semidisk, and the dynamical billiards on a semidisk is completely integrable.
\end{remark}

\begin{remark}
Note that $ R(\phi_{\ast}) \to \infty$ when $\phi_{\ast} \to 0$.
In this case, better reduction schemes exist \cite{BZZ}, 
and the estimates on $R(\phi_{\ast})$ could be significantly improved. 
\end{remark}

\begin{remark}
Recall the strictly convex scattering condition formulated by Wojtkowski \cite{Woj86}:
\begin{align}\label{Wcondition}
\tau(x) > d(x) + d(Fx).
\end{align}
See Section \ref{prelim} for the definitions of these notations.
It is interesting to note that the  exactly opposite 
inequality holds for asymmetric lemon billiards: $\tau(x)\le d(x) + d(Fx)$. 
This is due to the fact that $Q(r,b,R)$ is the {\it intersection} of two disks. 
\end{remark}

\section{Preliminaries}
\label{prelim}

In this section we give some preliminary results about asymmetric lemon billiards.
See \cite{CMZZ, BZZ} for more details. For general planar billiards, see \cite{CM06}.
 
 \subsection{The asymmetric lemons}
Let $Q(r,b,R)=D(O_r,r) \cap D(O_R, R)$ be the intersection of two disks of radii 
$r$ and $R$, respectively, where $b=|O_rO_R|$ is the distance between the two centers. 
The boundary $\Gamma =\pa Q(r,b,R)$ consists 
of two circular arcs $\Gamma_r$ and $\Gamma_R$,
with two corners at $A, B \in \pa D(O_r,r) \cap \pa D(O_R, R)$. 
We assume $R-r<b < R+r$ so that the intersection $Q(r,b,R)$ is a nontrivial (asymmetric) lemon.
Note that the orbit $\cO(2)$ passing through the two centers $O_r$ and $O_R$ is periodic
of period $2$. It follows from \cite{Woj86,DOP03,OP05}
that this orbit is elliptic and nonlinearly stable if $b<r$ or $b>R$, parabolic if $b=r$ or $b=R$,
and hyperbolic if $r<b<R$. 
So $r\le b \le R$ is a necessary condition for the asymmetric lemon billiards $Q(r,b,R)$ 
to be hyperbolic.

\subsection{The phase space}
To describe the phase space of asymmetric lemon billiards, we first parametrize the 
boundary $\Gamma =\Gamma_r \cup \Gamma_R$. For each point $P\in \Gamma_r$,
we let $\phi(P)\in \bT=\bR/2\pi$ be the angle from the vector $\overrightarrow{O_RO_r}$ 
to the vector $\overrightarrow{O_rP}$ (counterclockwise oriented). 
Similarly, for each point $P\in \Gamma_R$,
we let $\phi(P)\in \bT=\bR/2\pi$ be the angle from the vector $\overrightarrow{O_RO_r}$ 
to the vector $\overrightarrow{O_RP}$ (counterclockwise oriented).
Then we have $\phi(A) = \phi_{\ast}$, $\phi(B)= 2\pi -\phi_{\ast}$,
$\Phi(A) = \Phi_{\ast}$ and $\Phi(B)= -\Phi_{\ast}$. See Fig.~\ref{asl}.
Both corners $A$ and $B$ will be treated as   points on $\Gamma_r$. 
It follows that that $\Gamma_r=[\phi_{\ast}, 2\pi -\phi_{\ast}]$, $\Gamma_R=(-\Phi_{\ast}, \Phi_{\ast})$,
and $\Gamma=[\phi_{\ast}, 2\pi -\phi_{\ast}] \sqcup (-\Phi_{\ast}, \Phi_{\ast})$ (a disjoint union).

Let $T_{\Gamma}\bR^2$ be the set of tangent vectors over points in $\Gamma$.
The phase space $M\subset T_{\Gamma}\bR^2$ of  the asymmetric lemon billiards consists of 
unit vectors $x\in T_{\Gamma}\bR^2$ that point to the inside of the table $Q(r,b,R)$.
Let $p: M\to \Gamma$ be the projection from $M$ to $\Gamma$.
For each $x\in M$, let $\phi(x)\in [\phi_{\ast}, 2\pi -\phi_{\ast}]\sqcup (-\Phi_{\ast}, \Phi_{\ast})$ be  the position coordinate 
of $p(x)\in\Gamma$, and let $\theta(x) \in (0,\pi)$ be the angle 
from the positive tangent direction of $\Gamma$ at $p(x)$ to $x$.
By identifying $x$ with $(\phi(x),\theta(x))$,
we get a parametrization of the phase space 
$M=M_r \sqcup M_R$,
where $M_r=[\phi_{\ast}, 2\pi -\phi_{\ast}] \times  (0,\pi)$ and $M_R=(-\Phi_{\ast}, \Phi_{\ast}) \times  (0,\pi)$.

\subsection{The billiard map}\label{defmap}
Let $(\phi_0, \theta_0)\in M$. This corresponds to a unit vector $x_0\in T_{\Gamma}\bR^2$ pointing
to the inside of $Q(r,b,R)$. 
Suppose the ray $\bR_+\langle x_0 \rangle$ crosses $\Gamma$ at a point other than the two corners, 
say $\phi_1$. Then the ray make an elastic reflection with respect to the tangent line of $\Gamma$ at $\phi_1$. 
Let $\theta_1$
be the new direction coordinate with respect to the positive tangent direction of $\Gamma$ at $\phi_1$.
Then the map $F: M\to M$, $x_0=(\phi_0, \theta_0) \mapsto x_1=(\phi_1, \theta_1)$ is the billiard map on $M$. 

Note that the tangent bundle of $\Gamma$ is not continuous at the two corners $A$ and $B$.
Therefore, the map $F$ is not smooth (maybe even undefined)  if either $p(x_0)\in \{A, B\}$ or $p(x_1) \in  \{A, B\}$.
Let $\cS_1$ be the set of points $x\in M$ where $F$ is not smooth,
which is called the {\it singularity set} of $F$.
It is easy to see that $\cS_1$ consists of  $\{\phi_{\ast},  2\pi -\phi_{\ast}\}\times  (0,\pi)$  and
four skew segments in the interior of $M$ (two in $M_r$ and the other two in $M_R$).

For any $x_0=(\phi_0, \theta_0) \in M \backslash \cS_1$,
let $x_1=(\phi_1, \theta_1)=F x_0$,
$\tau(x_0)$ be the Euclidean distance from the initial point $p(x_0)$ to the terminal point $p(x_1)$, 
$r_i \in\{r,R\}$ be the radius of the arc containing $p(x_i)$, $i=0, 1$. 
Wojtkowski \cite{Woj86} introduce a function $d: M\to \bR$, 
where $d(x) = r\sin\theta$ if $x= (\phi, \theta) \in M_r$, and $d(x) = R\sin\theta$ if $x= (\phi, \theta) \in M_R$.  
The geometric meaning of $d(x)$ is the half length of the chord along the trajectory of
$x$ in the osculating circle of $\Gamma$ at $p(x)$.
Then the tangent map of the billiard map $F$ at a point $x\in M$ is given by
\begin{align}\label{tangent}
D_{x_0}F=\frac{1}{d(x_1)}
\begin{bmatrix} \tau(x_0) -d(x_0) & \tau(x_0) \\ 
\tau(x_0) -d(x_0) -d(x_1) & \tau(x_0) - d(x_1) \end{bmatrix}.
\end{align}
For example, if $p(x_i)$, $i=0,1$ are on one circular arc,
then $\tau(x_0) = 2d(x_0) =2d(x_1)$, and  $D_{x_0}F=\begin{bmatrix} 1 & 2 \\ 0 & 1 \end{bmatrix}$. 
Note that \eqref{Wcondition} is equivalent to that
all four entries of Eq.~\eqref{tangent} are positive.

The billiard map $F$ on $M$ preserves the 2-form $\omega=\rho(\phi) \sin\theta\, d\phi \wedge d\theta$,
$\rho(\phi)$ is the radius of curvature of $\pa Q$ at $\phi$.
Therefore, it preserves the corresponding probability measure $\mu$ on $M$, where 
\begin{align}
d\mu= C\cdot \rho(\phi) \sin\theta\, d\phi \,d\theta,
\end{align}
where $C=\frac{1}{2|\Gamma|}$ is a normalizing constant such that $\mu(M)=1$.
Note that $\mu(\cS_1)=0$. 

It follows from Oseledets {\it Multiplicative Ergodic Theorem} 
that  the limit $\ds \chi(x,F)=\lim_{n\to\infty}\frac{1}{n}\log\|D_xF^n\|$  exists for $\mu$-a.e. $x\in M$,
which is called  the {\it Lyapunov exponent} of the billiard map $F$ at $x$. 
Then $x\in M$ is a {\it hyperbolic point} of $F$ if $\chi(x,F)>0$,
and the dynamical billiards is said to be {\it hyperbolic} if $\mu$-a.e. $x\in M$ is 
a hyperbolic point for the billiard map $F$. 

\subsection{Time reversibility} 
Consider the map
$I: M\to M$, $(\phi, \theta) \mapsto (\phi, \pi-\theta)$. 
This is an involution since $I^2(\phi, \theta)=(\phi, \theta)$.
The billiard map $F$ is time-reversible. That is,
$F^{-1}\circ I = I\circ F$.

\section{Defocusing segments and  Hyperbolicity}
\label{proofofmain}

Let $\phi_{\ast }\in (0, \frac{\pi}{2})$, 
$Q(\phi_{\ast},R)$ be the asymmetric lemon table such that two arcs $\Gamma_r$
and $\Gamma_R$ intersect at the two points $A$ and $B$ with coordinates $\phi(A)=\phi_{\ast}$
and  $\phi(B)= 2\pi -\phi_{\ast}$, respectively.
Note that the boundary component $\Gamma_r$ is a major arc.
We  mainly use the $\Gamma_r$-part $M_r\subset M$ of the phase  space of the billiards $Q(\phi_{\ast},R)$,
as this part stays unchanged when we adjust the value $R$.

We will introduce  a subset $\hM\subset M_r$ and consider the first return map $\hF $ of 
the billiard map $F$ with respect to $\hM$ in \S \ref{return}.
Some preparation is need to define this subset $\hM$. We start with two subsets of $M_r$:

\begin{enumerate}
\item $M_r^{in}:=M_r \cap FM_R$, which is the set of points $x\in M_r$ with $F^{-1}x \in M_R$;

\item $M_r^{out}:=M_r \cap F^{-1}M_R$, which is the set of points $x\in M_r$ with $Fx \in M_R$.
\end{enumerate}
It is easy to see that $I(M_r^{in})= M_r^{out}$ due to the time-reversibility  of the billiard map $F$.
We define the subsets $M_R^{in}, M_R^{out}\subset M_R$ in a similar way.
Then  $I(M_R^{in})= M_R^{out}$.
See Fig.~\ref{partition}.

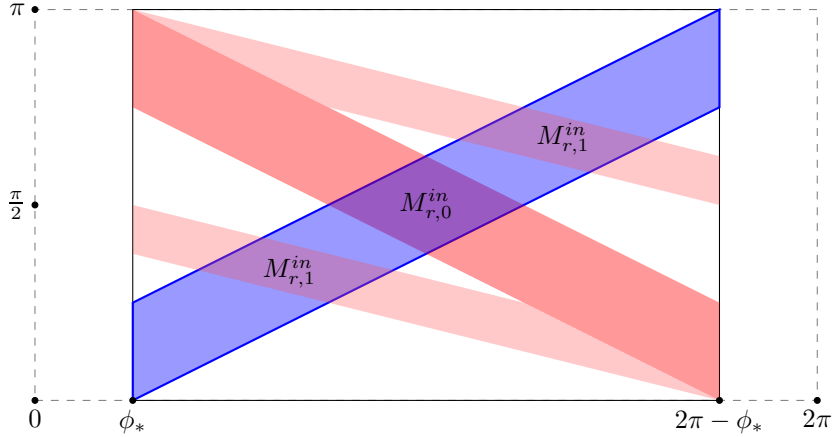
\begin{figure}[h]
\tikzmath{
\a=1;
\b=4;
\c=2*\b-\a;
}
\begin{tikzpicture}[scale=1.3]
\draw[dashed,gray] (0,0) rectangle (2*\b,\b);
\draw (\a,0) rectangle (\c, \b);
\draw[thick, blue] (\a,0) -- (\c, \b - \a) -- (\c, \b) -- (\a, \a) -- cycle;
\fill[red, fill opacity=0.4] (\a,\b) -- (\c, \a) -- (\c, 0) -- (\a,\b - \a) -- cycle;
\fill[blue, fill opacity=0.4] (\a,0) -- (\c, \b - \a) -- (\c, \b) -- (\a, \a) -- cycle;
\fill[red!70, fill opacity=0.3]  (\c, 0) -- (5, 1) -- (\a,2) -- (\a, 1.5) -- cycle;
\fill[red!70, fill opacity=0.3]  (\a, \b) -- (3, 3) -- (\c,2) -- (\c, 2.5) -- cycle;
\node (M0) at (\b,{0.5*\b}){$M_{r,0}^{in}$};
\node (M1a) at ({0.65*\b},{0.33*\b}){$M_{r,1}^{in}$};
\node (M1b) at ({1.35*\b},{0.67*\b}){$M_{r,1}^{in}$};
\filldraw (0,0) circle(.03)  node[below] {$0$};
\filldraw (2*\b,0) circle(.03)  node[below] {$2\pi$};
\filldraw (0,\b) circle(.03)  node[left] {$\pi$};
\filldraw (0,0.5*\b) circle(.03)  node[left] {$\frac{\pi}{2}$};
\filldraw (\a, 0) circle(.03)  node[below] {$\phi_{\ast}$};
\filldraw (\c, 0) circle(.03)  node[below] {$ 2\pi -\phi_{\ast}$};
\end{tikzpicture}
\caption{The set $M_r \subset \bR/2\pi\times [0,\pi]$.
The blue parallelogram is $M_r^{in}$, and the  red parallelogram is $M_{r}^{out}$.}
\label{partition}
\end{figure}

Let $x\in M_r^{in}$, $n(x)=\inf\{n\ge 0: F^{n} x\in M_r^{out}\}$, which represents 
the number of {\it remaining} reflections the orbit of $x$ has on $\Gamma_r$.
Let $M_{r,n}^{in}=\{x\in M_r^{in}: n(x)=n\}$, $n\ge 0$. 
This provides a partition of $M_r^{in}$. 
It follows from the definition that $M_{r,n}^{out} :=F^n (M_{r,n}^{in}) \subset M_{r}^{out}$ for each $n\ge 0$,
and together they form a partition of $M_{r}^{out}$.
Note that 
\begin{itemize}
\item $M_{r,n}^{out} =I(M_{r,n}^{in})$ by the time reversibility of the billiard map;

\item $M_{r,0}^{in} = M_{r,0}^{out} =M_r^{in}\cap M_r^{out}$;

\item $M_{r,n}^{in}$ has at least two connected components for each $n\ge 1$.
\end{itemize}

\subsection{Frequently used notations}\label{shortn}
The following notations will be used throughout the rest of the paper.
Given a point $x\in M_r$,
let $n_0=\inf\{n\ge 0: F^{n}x\in M_r^{out}\}$ be the number of remaining reflections
of the point $x$ has on $\Gamma_r$. Note that $n_0<\infty$ for every $x\in M_r$
except  finitely many segments of periodic points that never leave $\Gamma_r$.
Let
$x_0=(\phi_0 ,\theta_0):= F^{n_0}x\in M_r^{out}$, and  $x_1=(\phi_1, \theta_1):= F x_0  \in M_R^{in}$.
Let $n_1=\inf\{n\ge 0: F^{n}x_1 \in M_R^{out}\}$ be the number of remaining reflections
of $x_1$ has on $\Gamma_R$.
Then $F^{n_1}x_1\in M_R^{out}$ and
$x_2=(\phi_2 ,\theta_2):= F^{n_1+1}x_1\in M_r^{in}$.
Let $\tau_0$ be the distance from $p(x_0) \in \Gamma_r$ to $p(x_{1})\in\Gamma_R$,
and  $\tau_1$ be the distance from $p(F^{n_1}x_1) \in\Gamma_R$ to $p(x_{2})\in \Gamma_r$.
Let $d_0 =d(x_0) = r\sin\theta_0$, $d_1=d(x_1) = R \sin \theta_1$
and $d_2 =d(x_2) = r\sin\theta_2$. Then
$(x,\dots, x_0, x_1, \dots, F^{n_1}x_1, x_2)$ is an orbit segment of the billiard map $F$.
We have suppressed the dependence of 
these objects  on the point $x\in M_r$.

\subsection{Small neighborhood of points whose trajectories are close to
the chord}\label{nbhd}
Let $x_{\ast}=( 2\pi -\phi_{\ast}, \phi_{\ast})$ and $y_{\ast}=(\phi_{\ast}, \pi-\phi_{\ast})$ be two points in $M_r^{out}$ 
whose trajectories coincide with the chord $AB$. 
Then the two points $I y_{\ast}= (\phi_{\ast}, \phi_{\ast})$
and $I x_{\ast}= ( 2\pi -\phi_{\ast}, \pi-\phi_{\ast})$ are  in $M_r^{in}$. See Fig.~\ref{UVdelta}.

\begin{figure}[h]
\tikzmath{
\a=1;
\b=4;
\c=2*\b-\a;
}
\begin{tikzpicture}[scale=1.3]
\draw[gray, dashed] (0,0) rectangle (2*\b,\b);
\draw (\a,0) rectangle (\c, \b);
\filldraw[red, fill opacity=0.4] (\a,\b) -- (\c, \a) -- (\c, 0) -- (\a,\b - \a) -- cycle;
\filldraw[blue, fill opacity=0.4] (\a,0) -- (\c, \b - \a) -- (\c, \b) -- (\a, \a) -- cycle;
\coordinate (A0) at (\a,0);
\coordinate (B0) at (\c,0);
\coordinate (A1) at (\a,\b);
\coordinate (B1) at (\c,\b);
\filldraw (0,0) circle(.03)  node[below] {$0$};
\filldraw (0,\b) circle(.03)  node[left] {$\pi$};
\filldraw (0,0.5*\b) circle(.03)  node[left] {$\frac{\pi}{2}$};
\filldraw (2*\b,0) circle(.03)  node[below] {$2\pi$};
\filldraw (A0) circle(.03)  node[below] {$\phi_{\ast}$};
\filldraw (B0) circle(.03)  node[below] {$ 2\pi -\phi_{\ast}$};
\coordinate (x1) at (\c, \a);
\coordinate (y1) at (\a, \b - \a);
\coordinate (x2) at (\a,\a);
\coordinate (y2) at (\c, \b -\a);
\fill (x1) circle(.03)  node[right] {$x_{\ast}$};
\fill (y1) circle(.03)  node[left] {$y_{\ast}$};
\fill (x2) circle(.03)  node[left] {$I y_{\ast}$};
\fill (y2) circle(.03)  node[right] {$I x_{\ast}$};
\fill[red] (x1) -- (\c, \a - .4) arc (270: 154: .4) -- cycle;
\fill[red] (y1) -- (\a, \b - \a + .4) arc (90: -26: .4) -- cycle;
\fill[blue] (y2) -- (\c, \b -\a + .4) arc (90: 210: .4) -- cycle;
\fill[blue] (x2) -- (\a, \a  - .4) arc (-90: 26: .4) -- cycle;
\end{tikzpicture}
\caption{The blue region is $M_{r}^{in}$ and the red region is $M_{r}^{out}$.
The set $U(\delta)$ is the union of the two blue sectors, and  $V(\delta)$ is the union of two red sectors.}
\label{UVdelta}
\end{figure}
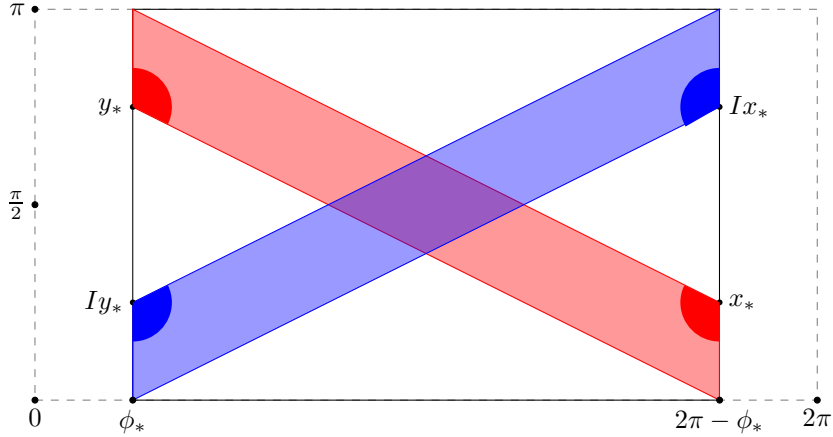

Given $\delta>0$,
let $V(x_{\ast},\delta) = B(x_{\ast}, \delta) \cap M_{r}^{out}$,
$V(y_{\ast},\delta) = B(y_{\ast}, \delta) \cap M_{r}^{out}$, and
$V(\delta) = V(x_{\ast},\delta) \cup V(y_{\ast},\delta)$ be 
the  $\delta$-neighborhood of 
$\{x_{\ast}, y_{\ast}\}$ in $M_{r}^{out}$.
Similarly, let $U(I x_{\ast},\delta) = B(I x_{\ast}, \delta) \cap M_{r}^{in}$,
$U(I y_{\ast},\delta) = B(I y_{\ast}, \delta) \cap M_{r}^{in}$, and
$U(\delta) = U(I x_{\ast},\delta) \cup U(I y_{\ast},\delta)$ be 
the  $\delta$-neighborhood of 
$\{I x_{\ast}, I y_{\ast}\}$ in $M_{r}^{in}$.
It is easy to see that $V(\delta)  = I(U(\delta))$.

\begin{pro}\label{UV}
Let $0<\delta \le \min\{\phi_{\ast}, \frac{\pi}{2}-\phi_{\ast}\}$ be given. Then we have
\begin{enumerate}
\item $U(\delta) \subset \bigcup_{n\ge 1} M_{r,n}^{in}$ and $V(\delta) \subset \bigcup_{n\ge 1} M_{r,n}^{out}$;

\item $U(\delta)\cap V(\delta) =\emptyset$ and $F(U(\delta)) \cap V(\delta)= \emptyset$.
\end{enumerate}
\end{pro}
\begin{proof}
Let $0< \delta\le  \min\{\phi_{\ast}, \frac{\pi}{2}-\phi_{\ast}\}$ be given. 
Then we have 
\begin{enumerate}
\item if  $x=(\phi,\theta) \in V(x_{\ast}, \delta)$, then
$2\pi -\phi_{\ast} -\delta <\phi < 2\pi -\phi_{\ast}$, $\phi_{\ast} -\delta < \theta< \phi_{\ast}+\delta$;

\item if  $x=(\phi,\theta) \in V(y_{\ast}, \delta)$, then
$\phi_{\ast}<\phi< \phi_{\ast} +\delta$, $\pi- \phi_{\ast}-\delta < \theta <\pi- \phi_{\ast}+\delta$; 

\item if  $x=(\phi,\theta) \in U(Ix_{\ast}, \delta)$, then
$2\pi -\phi_{\ast} -\delta <\phi < 2\pi -\phi_{\ast}$, $\pi- \phi_{\ast}-\delta < \theta <\pi- \phi_{\ast}+\delta$; 

\item if  $x=(\phi,\theta) \in U(Iy_{\ast}, \delta)$, then
$\phi_{\ast}<\phi< \phi_{\ast} +\delta$, $\phi_{\ast} -\delta < \theta< \phi_{\ast}+\delta$.
\end{enumerate}

Given a point $x=(\phi,\theta) \in U(Iy_{\ast}, \delta)$, we have $\phi_{\ast}<\phi< \phi_{\ast} +\delta$
and $\phi_{\ast}<3\phi_{\ast} - 2\delta<\phi + 2\theta < 3\phi_{\ast} +3\delta < 2\pi -\phi_{\ast} -\delta$,
since $\delta \le  \min\{\phi_{\ast}, \frac{\pi}{2} -\phi_{\ast}\}$.
It follows that $U(Iy_{\ast}, \delta) \subset \bigcup_{n\ge 1} M_{r,n}^{in}$, 
$U(Iy_{\ast}, \delta)\cap V(\delta) =\emptyset$ and $F(U(Iy_{\ast}, \delta)) \cap V(\delta)= \emptyset$.
The same conclusions hold for the part $U(Ix_{\ast}, \delta)$. This completes the proof.
\end{proof}

In the following we will set $\delta=\delta_\ast:= \min\{\phi_{\ast}, \frac{\pi}{2}-\phi_{\ast}\}$,
and consider the sets $ U(\delta_{\ast})$ and $ V(\delta_{\ast})$.

To prove Theorem \ref{mainthm}, we will consider orbit segments that start on $\Gamma_r$,
have some reflections on $\Gamma_R$ and then return to $\Gamma_r$,
Let $(x,\dots, x_0, x_1,\dots, F^{n_1}x_1, x_2)$ be such a segment.
See Section \ref{shortn} for these notations.
The following two cases will be treated separately: 
1). the case that there is only one reflection on $\Gamma_R$ (that is, when $n_1 =0$); 
and 2). the case that there are multiple reflections on  $\Gamma_R$ (that is, when $n_1 \ge 1$). 
Then  we 
divide each case into several subcases. There are two subcases
when $n_1=0$: 
\begin{enumerate}
\item[1a).] $d_1 = R\sin\theta_1\ge 2r$, 
which means the orbit segment is uniformly transverse to $\Gamma_R$ at $P_1=p(x_1)$;

\item[1b).]$d_1< 2r$, which means  the orbit segment is almost tangent to $\Gamma_R$ at $P_1=p(x_1)$. 
\end{enumerate}
Note that $d_1< r\sin\phi_{\ast}$ when $n_1\ge 1$.
 The following lemma describes the patterns of orbit segments that are almost tangent to $\Gamma_R$.

\begin{lem}\label{smalld1}
Let $\phi_{\ast}\in(0,\pi/2)$ be fixed, $\delta_{\ast}=\min\{\phi_{\ast}, \frac{\pi}{2}-\phi_{\ast} \}$,
and $U(\delta_{\ast})$ and $V(\delta_{\ast})$ be the open subsets of $M_r$
given in Proposition \ref{UV}.
Then for any $R \ge \max\{\frac{34r}{\phi_{\ast}},\frac{14.6r}{\delta_{\ast}\cdot \sin\phi_{\ast}} \}$, 
the following holds for the billiard map
on $Q(\phi_{\ast},R)$: for any $x \in M_r$, if $d_1< 2r$,
then $x_0\in V(\delta_{\ast})$, and  $x_2 \in U(\delta_{\ast})$.
\end{lem}
See Section \ref{shortn} for the definitions of the points $x_i=(\phi_i, \theta_i)$, $i=0, 1, 2$ and the quantity $d_1$.
The proof of Lemma \ref{smalld1} will be given at \S \ref{pf.smalld1}.

\subsection{The first return map}\label{return}
We will consider the following shuffled version of $M_r^{in}$ using the partition
$M_r^{in}=\bigcup_{n\ge 0}M_{r,n}^{in}$:
\begin{align}
\hM= M_{r,0}^{in} \cup \Big(M_{r,1}^{in} \backslash U(\delta_{\ast}) \Big)\cup 
F\big(M_{r,1}^{in} \cap U(\delta_{\ast})\big) \cup \bigcup_{n\ge 2}FM_{r,n}^{in}.
\label{hM}
\end{align}
In the case when $U(\delta_{\ast})\cap  M_{r,1}^{in} =\emptyset$, 
the  definition \eqref{hM} reduces to 
\begin{align}
\hM= M_{r,0}^{in} \cup M_{r,1}^{in} \cup \bigcup_{n\ge 2}FM_{r,n}^{in}.
\end{align}

For each $x\in \hM$, let $\sigma(x)=\inf\{n\ge 1: F^n (x) \in \hM\}$ be the first return time 
of $x$ to $\hM$, and  $\hF: \hM \to \hM$, $x\mapsto F^{\sigma(x)} x$ 
be the first return map of $F$ on $\hM$.

Note that the orbit segment  $(x_0, x_1, x_2)$ when $n_1=0$ 
(or the segment $(x_0, x_1, \dots, F^{n_1} x_1, x_2)$ when $n_1\ge 1$) 
is a subsegment of $(F^k x)_{0\le k \le \sigma(x)}$.
We need a finer description of orbit segments $(F^k x)_{0\le k \le \sigma(x)}$
if $d_1<2r$.

\begin{pro}\label{extension}
Let $\phi_{\ast}\in(0,\pi/2)$ be fixed. Then for any 
$R \ge \max\{\frac{34r}{\phi_{\ast}},\frac{14.6r}{\delta_{\ast}\cdot \sin\phi_{\ast}} \}$, 
 the following holds for the billiard map $F$ on $Q(\phi_{\ast},R)$. 
Given a point $x\in \hM$, let  $x_i$, $i=0, 1, 2$ be given
as in Section \ref{shortn}. If $d_1< 2r$,  
then the segment $(F^{-1} x_0, x_0, \dots, x_2, F x_2)$ 
is a subsegment of the orbit segment $(F^k x)_{0\le k \le \sigma(x)}$.
\end{pro}
\begin{proof}
Let $R \ge \max\{\frac{34r}{\phi_{\ast}},\frac{14.6r}{(\frac{\pi}{2}-\phi_{\ast})\cdot \sin\phi_{\ast}} \}$ be fixed. 
Let $\hM$ be given by \eqref{hM}.
For each $x\in \hM$, let  $x_i$, $i=0, 1, 2$ be given
as in Section \ref{shortn}. Suppose $d_1 <2r$. 
Then it follows from
Lemma \ref{smalld1}  that $x_0 \in V(\delta_{\ast}) \subset \bigcup_{n\ge 1} M^{out}_{r,n}$ 
and $x_2\in U(\delta_{\ast}) \subset \bigcup_{n\ge 1} M^{in}_{r,n}$.  
It follows from our definition of $\widehat{M}$
in \eqref{hM} that $x_2 \notin \hM$ and $Fx_2 \in \widehat{M}$.  
For $x_0$, we divide it into the following two cases:

\noindent{\bf Case 1.} $x_0 \in V(\delta_{\ast})\cap M^{out}_{r,1}$.
Then $F^{-1}x_0 \in M^{in}_{r,1}$.
Since $F(U(\delta_{\ast})) \cap V(\delta_{\ast}) =\emptyset$, we have 
\[ F^{-1} x_0\in M^{in}_{r,1} \cap F^{-1}V(\delta_{\ast})  \subset M^{in}_{r,1} \backslash U(\delta_{\ast}) \subset \hM.\]
Therefore, $x= F^{-1} x_0$, and the segment
$(x, x_0, \dots, x_2, F x_2)$ coincides with  $(F^k x)_{0\le k \le \sigma(x)}$.

\noindent{\bf Case 2.} $x_0 \in V(\delta_{\ast})\cap M^{out}_{r,n}$ for some $n\ge 2$.
Then $F^{-n}x_n \in M^{in}_{r,n}$, and hence $x = F^{1-n} x_0 \in FM^{in}_{r, n} \subset \hM$. It follows that 
$(F^{-1} x_0, x_0, \dots, x_2, F x_2)$ 
is a subsegment of $(F^k x)_{0\le k \le \sigma(x)}$.

This completes the proof of the proposition.
\end{proof}

\subsection{Orbit segments with invariant cones}
We identify the tangent space $T_xM$  with $\bR^2$ 
via the mapping  $u\pa_\phi + v\pa_\theta \mapsto (u,v)$, for every $x\in M$.
Then a cone in $\bR^2$ can be viewed as a cone at a point $x\in M$,
or a constant cone-field on a subset $E \subset M$, depending on the context. 
\begin{definition}
An orbit segment $(F^k x)_{m\le k \le n}$ is said to have {\it  positive derivative} if all four entries
of the tangent map $DF^{n-m}: T_{F^m x}M \to T_{F^n x}M$ are positive.
It is said to have {\it  negative derivative} if all four entries
of the tangent map $DF^{n-m}: T_{F^m x}M \to T_{F^n x}M$ are negative.
\end{definition}
Note that if an orbit segment $(F^k x)_{m\le k \le n}$ has either positive derivative or negative derivative,
then the cone $\cC=\{(u, v) \in \bR^2: uv \ge 0 \}$ is  strictly invariant from $F^m x$ to $F^n x$.
That is, $DF^{n-m}(\cC) \subset \cC^o$, 
where  $\cC^o=\{(u, v) \in \bR^2: uv > 0 \} \cup \{(0,0)\}$ is the interior of $\cC$.
For example, the condition \eqref{Wcondition} implies
that the cone $\cC$ is (strictly) invariant under each iteration along a  (strictly) convex scattering arc.

The following is  a  simple observation:
\begin{pro}\label{concatenation}
Suppose the cone $\cC$ is (strictly) invariant along the orbit segment $(F^k x)_{m\le k \le n}$.
If the points $p(F^k x)$, $m'\le k \le m$ lie on one circular arc, 
and  the points $p(F^k x)$, $n\le k \le n'$ lie on one circular arc, 
then $\cC$ is (strictly) invariant along the concatenation $(F^k x)_{m'\le k \le n'}$.
\end{pro}
\begin{proof}
It suffices to note that $DF=\begin{bmatrix} 1 & 2 \\ 0 & 1 \end{bmatrix}$
whenever the points $p(F^k x)$ and $p(F^{k+1} x)$ lie on one circular arc.
\end{proof}

The following is our main proposition,
whose proof will be given in Section \ref{n1=0} and Section \ref{n1ge1}. 
\begin{pro}\label{defocusing}
Let $\phi_{\ast}\in(0,\pi/2)$ be fixed. 
Suppose $R\ge \max \big\{\frac{14.6r}{\delta_{\ast}\cdot \sin\phi_{\ast}}, 
\frac{147 r}{\sin^2\phi_{\ast}}, 1773.7 r\big\}$. 
Then the following holds for the billiard map
on $Q(\phi_{\ast},R)$: for each $x\in \hM$, let $x_i=(\phi_i, \theta_i)$ and $d_i(x)=d(x_i)$, $i=0, 1, 2$ 
be given as in Section \ref{shortn}. Then
\begin{enumerate}
\item if $n_1 =0$ and $d_1 \ge 2r$, then  the segment 
$(x_0, x_1, x_2)$ has negative derivative;

\item if $n_1 =0$ and $d_1 < 2r$, then  the segment 
  $(F^{-1} x_0, x_0, \dots, x_2, F x_2)$ has negative derivative;

\item if $n_1\ge 1$, 
then  the orbit segment  $(F^{-1} x_0, x_0, \dots, x_2, F x_2)$
has positive derivative.
\end{enumerate}
\end{pro}

This allows us to complete the proof of    Theorem \ref{mainthm}.
 
\begin{proof}[Proof of Theorem \ref{mainthm}]
Let $\phi_{\ast}\in(0,\pi/2)$, and 
$R\ge \max \big\{\frac{14.6r}{\delta_{\ast}\cdot \sin\phi_{\ast}}, \frac{147 r}{\sin^2\phi_{\ast}}, 1773.7 r\big\}$,
and $F$ be the billiard map on $Q(\phi_{\ast},R)$, $\hM\subset M_r$ be given in Section \ref{return}. 
Note that $R\ge  \frac{165 r}{\sin^2\phi_{\ast}} \ge \frac{33r}{\phi_{\ast}}$.
Let $\cC(x)=\{(u,v)\in T_xM: uv \ge 0\}$, $x\in\hM$, be the constant cone-field on $\hM$.
Combining Proposition \ref{extension}, \ref{concatenation}  and Proposition \ref{defocusing}, 
we see that the cone-field $\cC$ is strictly invariant along 
the orbit segment $(F^k x)_{0\le k\le \sigma(x)}$ for $\mu$-a.e. $x\in \hM$.
It follows from \cite{Woj85} that
the system $(\hM, \hF, \mu_{\hM})$ is hyperbolic. 
Since $\hM \cup F^{-1} \hM \supset M_r^{in}$, we see that $\bigcup_{n\in\bZ} F^n \hM =M$.
Therefore, $(M, F, \mu)$ is  hyperbolic.
This finishes the proof of Theorem \ref{mainthm}.
\end{proof}

\subsection{Proof of Lemma \ref{smalld1}}\label{pf.smalld1}
We will give some preliminary estimates.
Recall that $R\cos\Phi_{\ast} = b+ r\cos\phi_{\ast}$ and $R\sin\Phi_{\ast}=r\sin\phi_{\ast}$. 
Then 
\begin{align}
\Phi_{\ast}=\arcsin(\frac{r\sin\phi_{\ast}}{R})<1.002\cdot \frac{r\sin\phi_{\ast}}{R}, \label{estiPhiA}
\end{align}
for $R\ge 10r$. Here we have used that $10*\arcsin 0.1 = 1.00167.... < 1.002$.

For later convenience, we introduce another quantity $\Psi_R=\sin^{-1}\frac{2r}{R}$.
It is easy to see that
\begin{align}
\Psi_R=\arcsin(\frac{2r}{R})< 1.002 \cdot \frac{2r}{R}, \label{estiPsiA}
\end{align}
for $R\ge 20r$. 
It is clear that $\Psi_R > 2\Phi_{\ast}$, since $\sin\Psi_R =\frac{2r}{R}>2 \sin\Phi_{\ast}$.

\begin{proof}[Proof of Lemma  \ref{smalld1}]
Let  $U(\delta_{\ast})$ and $V(\delta_{\ast})$
be the open subsets of $M_r$   given in  Proposition \ref{UV},
$R \ge \max\{\frac{34r}{\phi_{\ast}},\frac{14.6r}{\delta_{\ast}\cdot \sin\phi_{\ast}} \}$ be fixed.
Given a point $x\in M_r$, and  $n_i$, $i=0, 1$, $x_i$, $i=0, 1, 2$ and $d_1$ be defined as 
in \S \ref{shortn}.
Suppose $d_1=R\sin\theta_1< 2r$, which is equivalent to $\theta_1 \in (0, \Psi_R) \cup (\pi -  \Psi_R , \pi)$.
We can do some reduction using the symmetries of asymmetric lemon billiards:

R1). The two cases $\theta_1 \in(0, \Psi_R)$ and  $\theta_1 \in(\pi - \Psi_R, \pi)$
are related to the symmetry of the billiard table with respect to the line through $O_r$ and $O_R$.
It suffices to consider the case with $\theta_1 \in(0, \Psi_R)$.

R2). Due to the time-reversal symmetry of the billiard map 
and the symmetry in the definition $V(\delta_{\ast}) = I(U(\delta_{\ast}))$,
it suffices to prove $x_0 \in V(\delta_{\ast})$.

We divide our analysis into two cases according to the number of reflections 
of the orbit segment on $\Gamma_R$:

Case 1. There is only one reflection on $\Gamma_R$.
Applying the above reductions we can assume
$\theta_1 \in(0, \Psi_R)$ and we only need to prove that $x_0 \in V(\delta_{\ast})$.
Let $\bar P$ be the point of intersection of the circle $\pa D(O_R, R)$
with  the line passing through $P_0=p(x_0)\in \Gamma_r$ and $P_1=p(x_1)\in \Gamma_R$.
Then the position angle $\bar\phi$ of $\bar P$ with respect to $O_R$
satisfies $\bar\phi <-\Phi_{\ast}$, since $\bar P$ lies outside of the arc $\Gamma_R$.
Let $Q$ be the perpendicular foot from $O_R$ to the line passing through  $P_0P_1$.
See Fig.~\ref{fig.psiA}.
Then the coordinates of the points $x_0=(\phi_0 ,\theta_0)\in M_r^{out}$ and 
$x_1 =Fx_0=(\phi_1 ,\theta_1)\in M_R$ are related in the following way:
\begin{align}
r\cos\theta_0 & = R\cos\theta_1 - b\cos(\theta_1 -\phi_1); \label{push1}\\
\phi_0 +\theta_0 &= -\angle_{O_r} (L,N) = -\angle_{O_R} (L,Q) = \phi_1 - \theta_1. \label{push2}
\end{align}

\begin{figure}[h]
\tikzmath{
\a=75;
\b=22.73; 
}
\centering
\begin{tikzpicture}
\coordinate (OR) at (0,0);
\coordinate (Or) at (4.1,0);
\node[right] at (6,0) {$L$};
\draw[gray,dashed] (OR) -- (6,0);
\filldraw (0,0) circle(.03) node[above] {$O_R$};
\filldraw (4.1,0) circle(.03) node[above] {$O_r$};
\coordinate (A) at ({5*cos(22.8)},{5*sin(22.8)});
\filldraw (A) circle(.03);
\draw (A) arc (\a:(360 - \a):2);
\draw (A) arc (\b: -40 :5);
\coordinate (P) at ({4.1+2*cos(-82)},{2*sin(-82)});
\filldraw (P) circle(.03)  node[left] {$P_0$};
\coordinate (bP) at ({5*cos(-37)},{5*sin(-37)});
\filldraw (bP) circle(.03) node[right] {$\bar P$};
\coordinate (P1) at ({5*cos(-5)},{5*sin(-5)});
\filldraw (P1) circle(.03)  node[right] {$P_1$};
\coordinate (L1) at ({5*cos(-5)+4.4*(cos(27)-cos(-5))},{5*sin(-5)+4.4*(sin(27)-sin(-5))});
\draw[->] (P) -- (P1)  -- (L1) node[pos=1.1] {$P_2$};
\draw[dashed]  (bP) -- (P);
\draw[blue!70] (0,0) -- ({4.8*cos(-21)},{4.8*sin(-21)});
\draw[blue!70] (Or) -- ({4.1+2*cos(-21)},{2*sin(-21)}) node[right] {$N$};
\node[right] at ({4.8*cos(-21)},{4.8*sin(-21)}){$Q$};
\draw[gray,dashed] (OR) -- (P1) (Or) -- (P);
\end{tikzpicture}
\caption{An orbit with one reflection on $\Gamma_R$. Here $P_0=p(x_0)$, $P_1=p(x_1)$, and $P_2 =p(x_2)$.
Both blue lines $O_RQ$ and $O_rN$ are perpendicular to the same line $P_0P_1$.}
\label{fig.psiA}
\end{figure}

Since $|\phi_1| < \Phi_{\ast}$ and $\theta_1 < \Psi_R$, 
we  have $\bar\phi =\phi_1 - 2\theta_1 \in( -\Phi_{\ast} - 2\Psi_R, -\Phi_{\ast})$. 
Comparing the positions of $P_0$ and $\bar P$, we get 
\begin{align}
b+r\cos\phi_0 & > R \cos \bar\phi > R\cos ( \Phi_{\ast} + 2\Psi_R)
= R\cos \Phi_{\ast} + R (\cos ( \Phi_{\ast} + 2\Psi_R) - \cos\Phi_{\ast}) \nonumber \\
& =  b+r\cos\phi_{\ast} - 2R\sin(\Phi_{\ast}+ \Psi_R)\cdot\sin\Psi_R
> b+r\cos\phi_{\ast}  - \frac{12 r^2}{R}, 
\end{align}
since $\sin\Psi_R=\frac{2r}{R}$, and $\sin(\Phi_{\ast}+ \Psi_R)< \sin\Phi_{\ast}+ \sin\Psi_R<\frac{3r}{R}$.
It follows from our assumption on $R$ that  $R>  \frac{14.6 r }{\cos\phi_{\ast}}$.
Then we have
\begin{enumerate}
\item $r\cos\phi_0> r\cos\phi_{\ast} - \frac{12 r^2}{R}> 0$.
Therefore, $- \frac{\pi}{2}< \phi_0 < -\phi_{\ast}$;

\item  $\ds 0 > \cos\phi_0 - \cos\phi_{\ast} > - \frac{12 r}{R}$.
\end{enumerate}
It follows from (1) that there exists $\phi_\ast \in (- \frac{\pi}{2}, -\phi_{\ast})$
such that 
\begin{align}
|\cos\phi_0 - \cos \phi_{\ast} | 
& = | \sin\phi_\ast| \cdot |\phi_0 - (- \phi_{\ast})| \ge \sin\phi_{\ast}  \cdot |\phi_0 + \phi_{\ast}|. \label{cossin}
\end{align}
Combining \eqref{cossin} with (2), we get
\begin{align} 
|\phi_0 + \phi_{\ast}| & \le \frac{1}{ \sin\phi_{\ast}} |\cos\phi_0 - \cos \phi_{\ast} |  < \frac{12 r}{R \sin\phi_{\ast}}.
\label{estiphi2}
\end{align}

Since $|\phi_1|< \Phi_{\ast}$ and $\theta_1\in(0,\Psi_R)$, it follows from \eqref{rbR} and \eqref{push1} that
\begin{align}
|r\cos\theta_0 - r\cos\phi_{\ast}|
&=\big|\big(R\cos\theta_1 - b\cos(\theta_1 -\phi_1)\big) 
- \big(R\cos\Phi_{\ast} - b\big)   \big|  \nonumber  \\
&\le R\cdot |\cos\theta_1 - \cos\Phi_{\ast} |  + b\cdot |\cos(\theta_1 -\phi_1)- 1|  \nonumber \\
&\le 2R\cdot \Big(|\sin\frac{\Phi_{\ast}-\theta_1}{2} \sin\frac{\Phi_{\ast}+\theta_1}{2} |  
+  \sin^2\frac{\theta_1 -\phi_1}{2}\Big)  \nonumber \\ 
&< 2R\cdot \Big(\frac{\Psi_R}{2} \frac{\Phi_{\ast}+\Psi_R}{2} 
+  \frac{(\Phi_{\ast}+\Psi_R)^2}{4}\Big)< \frac{7.54 r^2}{R}. \label{esti.theta0}
\end{align}

Since $\phi_0\in(-\frac{\pi}{2}, -\phi_{\ast})$, 
and $\phi_1-\theta_1 > -\Phi_{\ast} - \Psi_R> -3.006\cdot \frac{r}{R} \ge -\frac{1}{11}\phi_{\ast}$
for $R \ge \frac{34 r}{\phi_{\ast}}$, we have
\begin{align}
\theta_0=\phi_1-\theta_1 - \phi_0\in(\phi_{\ast}+ \phi_1-\theta_1,\frac{\pi}{2}+\phi_1-\theta_1)
\subset \Big(\frac{10}{11}\phi_{\ast}, \frac{\pi}{2}\Big).
\end{align}
Then there exists $\bar\theta\in (\frac{10}{11}\phi_{\ast}, \frac{\pi}{2})$ such that
\begin{align}
|\cos\theta_0 - \cos\phi_{\ast}|
& = |\sin(\bar \theta)\cdot (\theta_0 -\phi_{\ast})|
\ge \sin(\frac{10}{11}\phi_{\ast})\cdot |\theta_0 -\phi_{\ast}|
>\frac{10}{11}\sin \phi_{\ast} \cdot |\theta_0 -\phi_{\ast}|.
\end{align}
Combining it with  \eqref{esti.theta0}, we see that 
\begin{align}
|\theta_0 -\phi_{\ast}| < \frac{11}{10\sin\phi_{\ast}}\cdot |\cos\theta_0 - \cos\phi_{\ast}|
<  \frac{11}{10\sin\phi_{\ast}}\cdot \frac{7.54r}{R}
< \frac{8.3r}{R\sin\phi_{\ast}}.
\label{estitheta2}
\end{align}

Putting them together, we have 
\begin{align} 
\|x_0 - x_\ast\|=\big(|\phi_0 + \phi_{\ast}|^2 + |\theta_0 - \phi_{\ast}|^2\big)^{1/2}
< \frac{14.6r}{R\sin\phi_{\ast}}.
\end{align}
It follows that $\|x_0 - x_\ast\|< \delta_{\ast}$ for $R \ge \max\{\frac{34r}{\phi_{\ast}},\frac{14.6r}{\delta_{\ast}\cdot \sin\phi_{\ast}} \}$.
Since $x_0 \in M_r^{out}$, it follows that
$x_0 \in V(\delta_{\ast})$.

\begin{figure}[h]\tikzmath{
\a=75;
\b=22.73; 
}
\centering
\begin{tikzpicture}
\coordinate (OR) at (0,0);
\coordinate (Or) at (4.1,0);
\filldraw (0,0) circle(.03) node[above] {$O_R$};
\filldraw (4.1,0) circle(.03) node[above] {$O_r$};
\coordinate (A) at ({5*cos(22.8)},{5*sin(22.8)});
\draw (A) arc (\a:(360 - \a):2);
\draw (A) arc (\b: -\b :5);
\coordinate (P) at ({4.1+2*cos(-79)},{2*sin(-79)});
\filldraw (P) circle(.03);
\coordinate (P1) at ({5*cos(-12)},{5*sin(-12)});
\filldraw (P1) circle(.03) node[right] {$P_1$};
\coordinate (P2) at ({5*cos(12)},{5*sin(12)});
\filldraw (P2) circle(.03);
\coordinate (P3) at  ({4.1+2*cos(79)},{2*sin(79)});
\filldraw (P3) circle(.03);
\draw[->] (P) -- (P1)  node[pos=-.25] {$P_0$};
\draw[->] (P1) -- (P2);
\draw[->] (P2) -- (P3)  node[pos=1.25] {$P_2$};
\end{tikzpicture}
\caption{An orbit with two reflections on $\Gamma_R$. 
Here $P_0=p(x_0)$, $P_1=p(x_1)$, and $P_2=p(x_2)$.}\label{fig.n1}
\end{figure}
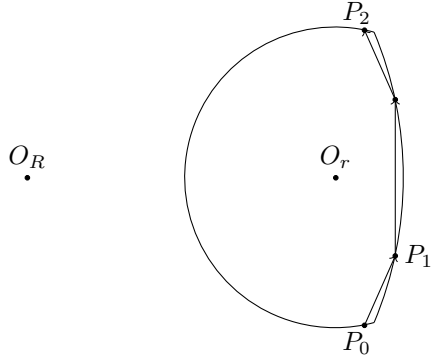

Case 2. There are two or more reflections on $\Gamma_R$. 
It follows that $\theta_1 \in (0, \Phi_{\ast}) \cup (\pi -  \Phi_{\ast} , \pi)$. 
Applying the reductions again, we assume $\theta_1 \in (0, \Phi_{\ast})$
and we only need to prove  $x_0 \in V(\delta_{\ast})$.
Let $\bar P$ be the point of intersection of the circle $\pa D(O_R, R)$
with  the line passing through $P_0=p(x_0)\in \Gamma_r$ and $P_1=p(x_1)\in \Gamma_R$.
The position angle $\bar\phi$ of $\bar P$ with respect to $O_R$
satisfies $-3\Phi_{\ast} < \bar\phi =\phi_1 - 2\theta_1 <  -\Phi_{\ast}$. 
Combining with \eqref{estiPhiA}, we get
\begin{align}
b+r\cos\phi_0 & > R \cos \bar\phi
> R\cos \Phi_{\ast} + R (\cos  3\Phi_{\ast}  - \cos \Phi_{\ast}) \nonumber \\
& =  b+r\cos\phi_{\ast} - 2R\sin2\Phi_{\ast} \sin\Phi_{\ast}
>  b+r\cos\phi_{\ast} - \frac{4 r^2\sin^2\phi_{\ast}}{R}.
\end{align}
It follows that $-\frac{\pi}{2}< \phi_0 < -\phi_{\ast}$, and hence
\begin{align} 
|\phi_0 - (-\phi_{\ast})| & \le \frac{1}{ \sin\phi_{\ast}} |\cos\phi_0 - \cos \phi_{\ast} |  < \frac{4 r \sin\phi_{\ast}}{R}.
\label{estiphi2n1}
\end{align}

Since $|\phi_1|< \Phi_{\ast}$ and $\theta_1\in(0, \Phi_{\ast})$, it follows from \eqref{rbR} and \eqref{push1} that
\begin{align}
|r\cos\theta_0 - r\cos\phi_{\ast}|
&=\big|\big(R\cos\theta_1 - b\cos(\theta_1 -\phi_1)\big) 
- \big(R\cos\Phi_{\ast} - b\big)   \big|  \nonumber  \\
&\le R\cdot |\cos\theta_1 - \cos\Phi_{\ast} |  + b\cdot |\cos(\theta_1 -\phi_1)- 1|  \nonumber  \\
&< 2R\cdot \frac{2r^2\sin^2\phi_{\ast}}{R^2}= \frac{4r^2\sin^2\phi_{\ast}}{R}.
\end{align}

Since $\phi_0\in(-\frac{\pi}{2}, -\phi_{\ast})$, and $\phi_1-\theta_1 > -2\Phi_{\ast}> -\frac{1}{17}\phi_{\ast}$
for $R \ge 34r$, we have
\begin{align}
\theta_0=\phi_1-\theta_1 - \phi_0\in(\phi_{\ast}+ \phi_1-\theta_1,\frac{\pi}{2}+\phi_1-\theta_1)
\subset \Big(\frac{16}{17}\phi_{\ast}, \frac{\pi}{2}\Big).
\end{align}
So there exists $\bar\theta\in (\frac{16}{17}\phi_{\ast},  \frac{\pi}{2})$ such that
\begin{align}
|\cos\theta_0 - \cos\phi_{\ast}|
& = |\sin(\bar \theta)\cdot (\theta_0 -\phi_{\ast})|
\ge \sin(\frac{16}{17}\phi_{\ast})\cdot |\theta_0 -\phi_{\ast}|
>\frac{16}{17}\sin \phi_{\ast} \cdot |\theta_0 -\phi_{\ast}|.
\end{align}
Combining them, we see that 
\begin{align}
|\theta_0 -\phi_{\ast}| < \frac{17}{16\sin\phi_{\ast}}\cdot |\cos\theta_0 - \cos\phi_{\ast}|
<  \frac{17}{16\sin\phi_{\ast}}\cdot  \frac{4r\sin^2\phi_{\ast}}{R}
= \frac{17r\sin\phi_{\ast}}{4R}.
\label{estitheta2n1}
\end{align}

Putting them together, we have 
\begin{align} 
\|x_0 -  x_\ast\|=\big(|\phi_0 + \phi_{\ast}|^2 + |\theta_0 - \phi_{\ast}|^2\big)^{1/2}
< \frac{5.84r\sin\phi_{\ast}}{R}.
\end{align}
It follows that $\|x_0 - x_\ast\|< \delta_{\ast}$ for $R \ge\max\{34r, \frac{5.84r \sin\phi_{\ast}}{\delta_{\ast}}\}$.
This completes the proof for the second case.

Collecting terms, we complete the proof of Lemma  \ref{smalld1}.
\end{proof}

\section{Orbit segments with a single reflection on $\Gamma_R$}
\label{n1=0}

In this section we will consider the case that the orbit segment  
$(F^{k}x)_{0\le k \le \sigma(x)}$ has exactly one reflection on $\Gamma_R$.
Let $x\in \hM$,  $x_i=(\phi_i, \theta_i)$ and $d_i(x)=d(x_i)$ for $i=0, 1, 2$,
$n_i$ for $i=0, 1$ be given in Section \ref{shortn}.
Note that $n_1=0$ in this section.
So the triple $(x_0, x_1, x_2)$ is part of the orbit segment $(F^{k}x)_{0\le k \le \sigma(x)}$.

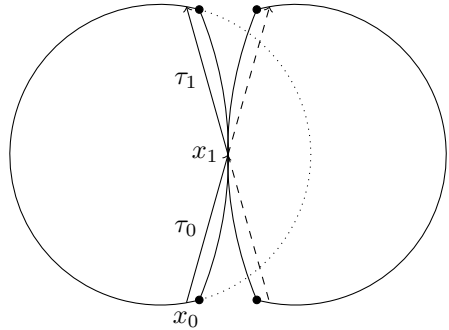
\begin{figure}[ht]
\tikzmath{
\a=75;
\b=22.73;
}
\begin{tikzpicture}
\coordinate (OR) at (0,0);
\coordinate (Or) at (4.1,0);
\coordinate (A) at ({4.1+2*cos(\a)},{2*sin(\a)});
\draw (A) arc (\b: -\b :5);
\draw (A) arc (\a:(360 - \a):2);
\filldraw (A) circle(.05);
\coordinate (B) at ({4.1+2*cos(-\a)},{2*sin(-\a)});
\filldraw (B) circle(.05);
\draw[dotted] (B) arc (-\a : \a :2);
\coordinate (C) at ({5.9+2*cos(180 - \a)},{2*sin(180 - \a)});
\draw (C) arc ((180 - \a) : (\a - 180) : 2);
\draw (C) arc ((180-\b):(180 + \b):5);
\filldraw (C) circle(.05);
\coordinate (D) at ({5.9+2*cos(\a - 180)},{2*sin(\a - 180)});
\filldraw (D) circle(.05);
\draw[->]  ({4.1+2*cos(-\a -5)},{2*sin(-\a -5)}) -- (5,0) node[pos=0,below]{$x_0$}  node[pos=0.5,left]{$\tau_0$}; 
\draw[->] (5,0)  -- ({4.1+2*cos(\a +5)},{2*sin(\a +5)})  node[pos=0,left]{$x_1$}  node[pos=0.5,left]{$\tau_1$};
\draw[dashed, ->]  (5,0) -- ({5.9+2*cos(180 - \a -5)},{2*sin(180 - \a -5)});
\draw[dashed]  (5,0) -- ({5.9+2*cos(-180 + \a +5)},{2*sin(-180 + \a +5)});
\end{tikzpicture}
\caption{The mirror table along the tangent line of $\Gamma_R$ at $p(x_1)$. }\label{maa}
\end{figure}

Let $\tau_i$ be the distance from $p(x_i)$ to $p(x_{i+1})$, $i=0, 1$.
By the major-arc assumption (see Fig.~ \ref{maa}), the union of the table with 
its mirror along the tangent line $T_{p(x_{1})}\Gamma_R$ covers the 
extended trajectory in $D(O_r, r)$. See the two dashed segments in Fig.~\ref{maa}.
It follows that
\begin{align}
\tau_0 +\tau_1 &> 2d_i, \quad i=0, 2.\label{majorarc0}
\end{align}
Combining the above inequalities for $i=0$ and $ i=2$, we get that
\begin{align}
\tau_0+\tau_1 &> d_0+d_2.\label{majorarc}
\end{align}

By \eqref{tangent}, the tangent map $D_{x_0}F^2: T_{x_0}M \to T_{x_2} M$ with respect to the $(\phi,\theta)$-coordinates
is given by
\begin{align}
D_{x_0}F^2
=&\frac{1}{d_2}\begin{bmatrix} \tau_1 -d_1 & \tau_1 \\ \tau_1-d_1-d_2 & \tau_1- d_2 \end{bmatrix}
\cdot\frac{1}{d_1}\begin{bmatrix} \tau_0 -d_0 & \tau_0 \\ \tau_0-d_0-d_1 & \tau_0- d_1 \end{bmatrix}.
\end{align}

Modulo the scalar $\frac{1}{d_2d_1}$, we get
\begin{align}
D= &\begin{bmatrix}D_{11} & D_{12} \\ D_{21} & D_{22} \end{bmatrix}:= (d_1d_2)\cdot D_{x_0}F^2 \nonumber \\
=&\begin{bmatrix} 
(\tau_1 -d_1)(\tau_0 -d_0) +\tau_1( \tau_0-d_0-d_1)
& (\tau_1 -d_1)\tau_0 +\tau_1( \tau_0-d_1) \\ 
(\tau_1 -d_1 -d_2)(\tau_0 -d_0) +(\tau_1-d_2)( \tau_0-d_0-d_1)
& (\tau_1 -d_1 -d_2)\tau_0 +(\tau_1-d_2)( \tau_0 -d_1)
\end{bmatrix}\label{DF2}\\
=&
\begin{bmatrix} 
2\tau_1(\tau_0 -d_0) -d_1(\tau_0 -d_0+\tau_1)
& 2\tau_1\tau_0  -d_1(\tau_0+\tau_1) \\ 
2(\tau_1-d_2)( \tau_0-d_0) - d_1(\tau_0-d_0+\tau_1-d_2)
& 2(\tau_1 -d_2)\tau_0 -d_1(\tau_0+\tau_1-d_2)
\end{bmatrix} \label{DF2d1}.
\end{align}

Since $\tau_0+\tau_1 > d_0+d_2$,
the coefficients of $d_1$ of all four entries of $D$ from \eqref{DF2d1} are negative.

\begin{lem}\label{d1large}
If $\frac{2}{d_1} < \frac{1}{\tau_0} + \frac{1}{\tau_1}$,
then the orbit segment $(x_0, x_1, x_2)$ has negative derivative.
\end{lem} 
\begin{proof}
Suppose $\frac{2}{d_1} < \frac{1}{\tau_0} + \frac{1}{\tau_1}$.
Note that this is equivalent to $D_{12}<0$. 
For the other three entries of $D_{x_0}F^2$,
we divide our analysis into the following cases:

\noindent{\bf Case 1.} $\tau_1-d_2>0$ and $\tau_0-d_0>0$:
In this case, $\frac{1}{\tau_0}<\frac{1}{\tau_0-d_0}$
and $\frac{1}{\tau_1}<\frac{1}{\tau_1-d_2}$.
Then 
\begin{enumerate}
\parskip=.1in
\item $\frac{2}{d_1} < \frac{1}{\tau_0-d_0} + \frac{1}{\tau_1}$
and hence $D_{11}=2\tau_1(\tau_0 -d_0) -d_1(\tau_0 -d_0+\tau_1)<0$;

\item $\frac{2}{d_1} < \frac{1}{\tau_0-d_0} + \frac{1}{\tau_1-d_2}$
and hence $D_{21}=2(\tau_1-d_2)(\tau_0 -d_0) -d_1(\tau_0 -d_0+\tau_1-d_2)<0$;

\item $\frac{2}{d_1} < \frac{1}{\tau_0} + \frac{1}{\tau_1-d_2}$
and hence $D_{22}=2(\tau_1 -d_2)\tau_0 -d_1(\tau_0+ \tau_1 -d_2)<0$.
\end{enumerate}

\noindent{\bf Case 2.}  $\tau_1-d_2>0$ and $\tau_0-d_0\le 0$:
Then $D_{22}<0$ since  $\frac{1}{\tau_1}<\frac{1}{\tau_1-d_2}$.
Moreover, $D_{11}$ and  $D_{21}$
are negative since both terms in $D_{11}$ and $D_{21}$ are negative, respectively.

\noindent{\bf Case 3.}  $\tau_1-d_2\le 0$ and $\tau_0-d_0>0$:
Then $D_{11}<0$ since $\frac{1}{\tau_0}<\frac{1}{\tau_0-d_0}$.
Moreover, $D_{21}$ and $D_{22}$
are negative since both terms in $D_{21}$ and $D_{22}$ are negative, respectively.

\noindent{\bf Case 4.} $\tau_1-d_2\le 0$ and $\tau_0-d_0\le 0$: this is impossible since
it contradicts  \eqref{majorarc}.

This completes the proof of the proposition.
\end{proof}

\begin{proof}[Proof of Proposition \ref{defocusing}.(1)]
Suppose $n_1=0$ and   $d_1\ge 2r$.
Since $Q(\phi_{\ast},R)$ is contained in the disk $D(O_r,r)$, we have $\tau_0< 2r$, and $\tau_1< 2r$.
It follows that $\frac{2}{d_1} \le \frac{1}{r} < \frac{1}{\tau_0} + \frac{1}{\tau_1}$.
Combining with Lemma \ref{d1large}, we see that
the orbit segment $(x_0, x_1, x_2)$ has  negative derivative.
This finishes the proof.
\end{proof}

Now we prove the second item in  Proposition \ref{defocusing}.
We will  show that $(F^{-1} x_0, x_0, x_1, x_2, F x_2)$  has negative derivative
when $n_1=0$ and $d_1<  2r$. In this case we have
$\theta_1\in (0, \Psi_R) \cup (\pi -\Psi_R, \pi)$.
Applying the symmetry of the billiard table $Q(\phi_{\ast},R)$,
it suffices to consider the case that $\theta_1\in (0, \Psi_R)$.

It follows from Lemma \ref{smalld1}, more precisely, from  Eq.~\eqref{estitheta2},
that
\begin{align}
|d_i - r\sin\phi_{\ast} | &= |r\sin \theta_i - r\sin\phi_{\ast}| < \frac{8.3 r^2}{R\sin \phi_{\ast}}, 
\quad i= 0, 2.
\label{estidi}
\end{align}
For $R\ge \frac{100r}{\sin^2\phi_{\ast}}$, we have  $\frac{8.3 r^2}{R\sin \phi_{\ast}}<\frac{1}{12}r\sin\phi_{\ast}$.
Therefore, $\frac{11}{12}r\sin\phi_{\ast}< d_i < \frac{13}{12}r\sin\phi_{\ast}$, $i=0,2$.

Note that the arc-length 
$|\Gamma_R|=2R\Phi_{\ast} < 2R\sin \Phi_{\ast} + 2R\cdot \frac{1}{3}\Phi_{\ast}^3 =2r\sin\phi_{\ast} + \frac{2R}{3}\Phi_{\ast}^3$.
Combining this with \eqref{estiPhiA} and \eqref{estiphi2}, we have
\begin{align}
 2r\sin\phi_{\ast}<\tau_0 +\tau_1 
& <  |\Gamma_R| + r\cdot |\phi_0 + \phi_{\ast}| + r\cdot |\phi_2 - \phi_{\ast} |  \nonumber   \\
& < 2r\sin\phi_{\ast} + \frac{2R}{3}\Phi_{\ast}^3 + \frac{12r^2}{R\sin \phi_{\ast}} + \frac{12r^2}{R\sin \phi_{\ast}} 
< 2r\sin\phi_{\ast} +  \frac{24.1 r^2}{R\sin \phi_{\ast}},\label{tau0tau1}
\end{align}
since $\frac{2R}{3}\Phi_{\ast}^3<  \frac{0.1 r^2}{R\sin \phi_{\ast}}$
for $R\ge 100r$.

Combining \eqref{majorarc0} and \ref{majorarc} with \eqref{estidi} and \eqref{tau0tau1}, we get that
\begin{align}
0<\tau_0 +\tau_1 - 2d_i  &< \frac{40.7 r^2}{R\sin \phi_{\ast}}, \quad i=0, 2; \label{tau0tau1di}\\
0<\tau_0 +\tau_1 - d_0 - d_2& <  \frac{40.7 r^2}{R\sin \phi_{\ast}}. \label{tau0tau1d0d2}
\end{align}

Let $D=(d_1d_2)D_{x_0}F^2$ be the matrix given by \eqref{DF2}, which corresponds to 
the tangent map along the  orbit segment $(x_0, x_1, x_2)$.
However,  the orbit segment $(x_0, x_1, x_2)$ may fail the negative derivative condition
for some point $x\in \hM$ with $d_1\le 2r$.
We need to consider the extended segment $(F^{-1}x_0, x_0, x_1, x_2, F x_2)$,
whose tangent map is (modulo the coefficient $\frac{1}{d_1d_2}$)
\begin{align}
G:=(d_1d_2)\cdot D_{F^{-1}x_0}F^4 =&\begin{bmatrix} 1 & 2 \\ 0 & 1 \end{bmatrix}
\cdot (d_1d_2) D_{x_0}F^2 
\cdot\begin{bmatrix} 1 & 2 \\ 0 & 1 \end{bmatrix}
=\begin{bmatrix} 1 & 2 \\ 0 & 1 \end{bmatrix}
\cdot \begin{bmatrix} D_{11} & D_{12} \\  D_{21} & D_{22} \end{bmatrix}
\cdot\begin{bmatrix} 1 & 2 \\ 0 & 1 \end{bmatrix}  \nonumber  \\
=&\begin{bmatrix} 
D_{11}+2D_{21}  & 2D_{11}+4D_{21} + D_{12}+2D_{22} \\  
D_{21} & 2D_{21}+ D_{22} 
\end{bmatrix}.
\end{align}

\begin{proof}[Proof of Theorem \ref{defocusing}.(2)]
We will show that all four entries of the matrix $G$ are  negative.
We will argue in the following order: the $(2,1)$-entry,  the $(1,1)$-entry, the $(2,2)$-entry and the $(1,2)$-entry.

\noindent{\bf  The $(2,1)$-entry.} Note that
\begin{align}
G_{21} = D_{21} 
=&2(\tau_1-d_2)(\tau_0-d_0)-d_1(\tau_0 +\tau_1 -d_0 -d_2). 
\end{align}
Note that $\tau_0 +\tau_1 -d_0 -d_2>0$. 
This term is clearly negative if $(\tau_1-d_2)(\tau_0-d_0)\le 0$. 
So we are left with the case
 $(\tau_1-d_2)(\tau_0-d_0)> 0$. Since $\tau_0 +\tau_1 -d_0 -d_2>0$, it follows
that $\tau_1-d_2>0$ and $\tau_0-d_0>0$. 
Since the billiard table is the intersection of two disks, we have 
$\tau_0< d_0 + d_1$ and $\tau_1< d_1 + d_2$. Putting them together, we have 
$0<\tau_0-d_0< d_1$, $0<\tau_1-d_2< d_1$ and hence 
\begin{align}
\frac{2}{d_1} = \frac{1}{d_1}+  \frac{1}{d_1} 
<\frac{1}{\tau_1-d_2} + \frac{1}{\tau_0-d_0}. 
\end{align}
It follows that $G_{21}<0$.

\noindent{\bf  The $(1,1)$-entry.}   Note that
\begin{align}
G_{11} =D_{11}+ 2 D_{21}
=&6(\tau_1- \frac{2}{3}d_2)(\tau_0-d_0)-3d_1(\tau_0 +\tau_1 -d_0 - \frac{2}{3}d_2).
\end{align}
This term is clearly negative if $(\tau_1- \frac{2}{3}d_2)(\tau_0 - d_0)\le 0$. So we are left with the case
 $(\tau_1-  \frac{2}{3}d_2)(\tau_0-d_0)> 0$, which implies
that $\tau_1- \frac{2}{3}d_2>0$ and $\tau_0-d_0>0$. 
We claim that $\tau_1< d_1 + \frac{2}{3}d_2$
when $R\ge \frac{128.6 r}{\sin^2\phi_{\ast}}$.
Then using $0<\tau_0-d_0< d_1$ again, 
we have 
\begin{align}
\frac{2}{d_1} = \frac{1}{d_1}+  \frac{1}{d_1}  <\frac{1}{\tau_1- \frac{2}{3}d_2} + \frac{1}{\tau_0-d_0}, 
\end{align}
and hence $G_{11}<0$ for $R\ge \frac{128.6 r}{\sin^2\phi_{\ast}}$.
\begin{proof}[Proof of Claim]
We will prove by contradiction. Suppose on the contrary that 
$\tau_1\ge d_1 + \frac{2}{3}d_2$. It follows that $d_1 > \frac{2}{3}d_2$ since $\tau_1< 2d_1$.
Combining with $\tau_0> d_0$ and  \eqref{tau0tau1d0d2}, we get
we get 
\begin{align}
d_0+ \frac{4}{3}d_2 < d_0+d_1 + \frac{2}{3}d_2
< \tau_0+\tau_1 < d_0+ d_2 + \frac{40.7 r^2}{R\sin\phi_{\ast}}. \nonumber 
\end{align}
Applying \eqref{estidi}, we get 
$\frac{1}{3}r\sin\phi_{\ast} -\frac{8.3 r^2}{3R\sin\phi_{\ast}} <\frac{1}{3}d_2<  \frac{40.7 r^2}{R\sin\phi_{\ast}}$,
which is impossible for $R\ge \frac{128.6 r}{\sin^2\phi_{\ast}}$.
This completes the proof.
\end{proof}

\noindent{\bf  The $(2,2)$-entry.}   Note that
\begin{align}
G_{22} =2 D_{21}+ D_{22}
=&6(\tau_1-d_2)(\tau_0-\frac{2}{3}d_0) - 3d_1(\tau_0 +\tau_1 -\frac{2}{3}d_0 -d_2).
\end{align}
This term is clearly negative if $(\tau_1-d_2)(\tau_0-\frac{2}{3}d_0)\le 0$. So we are left with the case
 $(\tau_1-d_2)(\tau_0 - \frac{2}{3}d_0)> 0$, which implies
that $\tau_1- d_2>0$ and $\tau_0- \frac{2}{3}d_0>0$. 
Following the same argument  used for the $(1,1)$-entry, 
we obtain  that $\tau_0< d_1 + \frac{2}{3}d_0$ for $R\ge \frac{128.6 r}{\sin^2\phi_{\ast}}$.
Then using $0<\tau_1-d_2< d_1$ again, 
we have 
\begin{align}
\frac{2}{d_1} = \frac{1}{d_1}+  \frac{1}{d_1}  <\frac{1}{\tau_1- d_2} + \frac{1}{\tau_0- \frac{2}{3}d_0},
\end{align}
and hence $G_{22}<0$ for $R\ge \frac{128.6 r}{\sin^2\phi_{\ast}}$.

\noindent{\bf  The $(1,2)$-entry.} Note that
\begin{align}
G_{12} &= 2D_{11}+4D_{21} + D_{12} + 2D_{22}  \nonumber \\
&=18(\tau_1-  \frac{2}{3}d_2)(\tau_0-  \frac{2}{3}d_0) - 9d_1(\tau_0 +\tau_1 - \frac{2}{3}d_0 - \frac{2}{3}d_2).
\end{align}
This term is clearly negative if $(\tau_1-  \frac{2}{3}d_2)(\tau_0-  \frac{2}{3}d_0)\le 0$. So we are left with the case
 $(\tau_1-  \frac{2}{3}d_2)(\tau_0-  \frac{2}{3}d_0)> 0$, which implies
that $\tau_1- \frac{2}{3}d_2>0$ and $\tau_0-\frac{2}{3}d_0>0$. 
In this case, $G_{12}<0$ if and only if 
\begin{align}
\frac{2}{d_1}  <\frac{1}{\tau_1- \frac{2}{3}d_2} + \frac{1}{\tau_0- \frac{2}{3}d_0}.
\label{n10G12}
\end{align}
We further divide our analysis into three subcases:

\noindent{\bf Case 1.} $\tau_0<\frac{2}{3}d_0+ d_1$ and $\tau_1< d_1 + \frac{2}{3}d_2$.
Then \eqref{n10G12} holds and hence $G_{12}<0$.

\noindent{\bf Case 2.} $\tau_0\ge \frac{2}{3}d_0+ d_1$. Combining with \eqref{tau0tau1d0d2},
we get $d_0+ d_2 + \frac{40.7 r^2}{R\sin\phi_{\ast}} > \tau_0 +\tau_1
\ge  \frac{2}{3}d_0+ d_1+ \tau_1$.
Therefore, $\tau_1 < \frac{1}{3}d_0+ d_2 -d_1 + \frac{40.7 r^2}{R\sin\phi_{\ast}}$. 
Combining $\tau_0\ge \frac{2}{3}d_0+ d_1$ with $2d_1>\tau_0$, 
we get $d_1> \frac{2}{3}d_0$.
Combining these with \eqref{estidi} for $d_0$, we have
\begin{align}
0<\tau_1- \frac{2}{3}d_2& < \frac{1}{3}d_0+ \frac{1}{3}d_2 -d_1 + \frac{40.7 r^2}{R\sin\phi_{\ast}}
<\frac{1}{3}d_2 -\frac{1}{3}d_0  + \frac{40.7 r^2}{R\sin\phi_{\ast}}  \nonumber   \\
&<\frac{16.6 r^2}{3R\sin\phi_{\ast}} + \frac{40.7 r^2}{R\sin\phi_{\ast}}
 \le \frac{1}{3}r\sin\phi_{\ast} -\frac{8.3 r^2}{3R\sin\phi_{\ast}}  <\frac{d_0}{3}< \frac{d_1}{2},
\end{align}
for $R\ge \frac{147 r}{\sin^2\phi_{\ast}}$.
Then  \eqref{n10G12} holds and hence  $G_{12}<0$ for $R\ge \frac{147 r}{\sin^2\phi_{\ast}}$.

\noindent{\bf Case 3.} $\tau_1\ge  d_1 + \frac{2}{3}d_2$. Using the same argument in Case 2,  we have
$0<\tau_0- \frac{2}{3}d_0<  \frac{d_1}{2}$
for $R\ge \frac{147 r}{\sin^2\phi_{\ast}}$. Then  \eqref{n10G12} holds and hence  $G_{12}<0$.

Collecting terms, we see that the orbit segment $(F^{-1}x_0, x_0, x_1, x_2, F x_2)$  
has negative derivative  for  $R\ge \frac{147 r}{\sin^2\phi_{\ast}}$.
This completes the proof of the second item of Proposition \ref{defocusing}.
\end{proof}

\section{Orbit segments with multiple reflections on $\Gamma_R$}\label{n1ge1}

In this section we consider points $x\in \hM$ whose orbit segments $(F^k x)_{0\le k \le \sigma(x)}$ 
have two or more  reflections on $\Gamma_R$.
We will  reuse most of the notations from Section \ref{n1=0}. 
Recall that $x_0=(\phi_0,\theta_0) \in M_{r}^{out}$, $x_1=(\phi_1,\theta_1)= F x_0 \in M_R^{in}$,
$n_1 \ge 1$ with $F^{n_1} x_1 \in M_{R}^{out}$,
$x_2=(\phi_2,\theta_2)=F^{n_1+1} x_1 \in M_{r}^{in}$,
$d_0=r\sin\theta_0$, $d_1=R\sin\theta_1$ and $d_2 = r\sin\theta_{2}$.
The intermediate points of this orbit segment on $\Gamma_R$ are 
$F^k x_1 =(\phi_1 + 2k \theta_1, \theta_1)$, $1\le k\le n_1$.
Let $\tau_0$ be the distance from $p(x_0)$ to $p(x_1)$,
$\tau_1$ be the distance from $p( F^{n_1} x_1)$ to $p(x_2)$. See Fig.~\ref{fig.n11} for an illustration.
We will show that the segment $(F^{-1}x_0, x_0, x_1,\dots, F^{n_1} x_1, x_2, F x_2)$
has positive derivative.

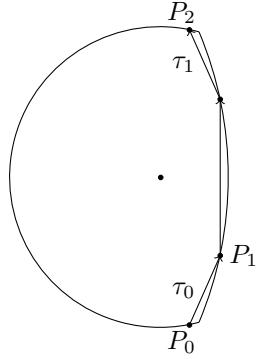
\begin{figure}[h]\tikzmath{
\a=75;
\b=22.73; 
}
\centering
\begin{tikzpicture}
\coordinate (OR) at (0,0);
\coordinate (Or) at (4.1,0);
\filldraw (4.1,0) circle(.03);
\coordinate (A) at ({5*cos(22.8)},{5*sin(22.8)});
\draw (A) arc (\a:(360 - \a):2);
\draw (A) arc (\b: -\b :5);
\coordinate (P) at ({4.1+2*cos(-79)},{2*sin(-79)});
\filldraw (P) circle(.03);
\coordinate (P1) at ({5*cos(-12)},{5*sin(-12)});
\filldraw (P1) circle(.03) node[right] {$P_1$};
\coordinate (P2) at ({5*cos(12)},{5*sin(12)});
\filldraw (P2) circle(.03);
\coordinate (P3) at  ({4.1+2*cos(79)},{2*sin(79)});
\filldraw (P3) circle(.03);
\draw[->] (P) -- (P1)  node[pos=-.25] {$P_0$} node[pos=0.5, left]{$\tau_0$} ;
\draw[->] (P1) -- (P2);
\draw[->] (P2) -- (P3)  node[pos=0.5,left]{$\tau_1$}  node[pos=1.25] {$P_2$};
\end{tikzpicture}
\caption{An illustration when $n_1=1$. Again $P_0=p(x_0)$, $P_1=p(x_1)$, and $P_2=p(x_2)$.}\label{fig.n11}
\end{figure}

Since the segment  $(x_1,\dots, F^{n_1}x_1)$ is on the same arc $\Gamma_R$, 
the tangent map along this segment is given by
\begin{align*}
D_{x_1}F^{n_1}=\Big(\begin{bmatrix} 1 & 2  \\ 0 & 1 \end{bmatrix}\Big)^{n_1}
=\begin{bmatrix} 1 & 2 n_1 \\ 0 & 1 \end{bmatrix}.
\end{align*}
Then the tangent map $D_{x_0}F^{n_1+2}$
along the orbit segment $(x_0, x_1,\dots, F^{n_1} x_1, x_2)$ is given by 
\begin{align*}
D_{x_0}F^{n_1+2}=DF\circ DF^{n_1} \circ DF
=\frac{1}{d_2}\begin{bmatrix} \tau_1 -d_1 & \tau_1 \\ \tau_1-d_1-d_2 & \tau_1- d_2 \end{bmatrix}
\cdot\begin{bmatrix} 1 & 2 n_1 \\ 0 & 1 \end{bmatrix}
\cdot\frac{1}{d_1}\begin{bmatrix} \tau_0 -d_0 & \tau_0 \\ \tau_0-d_0-d_1 & \tau_0- d_1 \end{bmatrix}.
\end{align*}

We introduce a matrix modulo the scalar $d_1d_2$:
\begin{align}
D:=&(d_1 d_2)D_{x_0}F^{n_1+2} 
=\begin{bmatrix} \tau_1 -d_1 & \tau_1 \\ \tau_1-d_1-d_2 & \tau_1- d_2 \end{bmatrix}
\Big(I+\begin{bmatrix} 0 & 2n_1 \\ 0 & 0 \end{bmatrix}\Big)
\cdot\begin{bmatrix} \tau_0 -d_0 & \tau_0 \\ \tau_0-d_0-d_1 & \tau_0- d_1 \end{bmatrix}  \nonumber  \\
=&\begin{bmatrix} 
(\tau_1 -d_1)(\tau_0 -d_0) +\tau_1( \tau_0-d_0-d_1)
& (\tau_1 -d_1)\tau_0 +\tau_1( \tau_0-d_1)   \nonumber  \\ 
(\tau_1 -d_1 -d_2)(\tau_0 -d_0) +(\tau_1-d_2)( \tau_0-d_0-d_1)
& (\tau_1 -d_1 -d_2)\tau_0 +(\tau_1-d_2)( \tau_0 -d_1)
\end{bmatrix}\\
+&2n_1
\begin{bmatrix} (\tau_1 -d_1)(\tau_0-d_0-d_1) 
& (\tau_1 -d_1)(\tau_0- d_1)  \\ 
(\tau_1-d_1-d_2)(\tau_0-d_0-d_1) 
& (\tau_1-d_1-d_2)(\tau_0- d_1) \end{bmatrix} .
\end{align}

Set $p=\frac{n_1}{n_1+1}$ for short.
Then the four entries of the matrix $D$ can be written as:
\begin{align}
D_{11}&=(\tau_1 -d_1)(\tau_0 -d_0) +\tau_1( \tau_0-d_0-d_1) +
2n_1(\tau_1 -d_1)(\tau_0-d_0-d_1)  \nonumber  \\
&=(n_1+1)\Big((\tau_1 -d_1)(\tau_0 - d_0 - pd_1)+(\tau_1-pd_1)( \tau_0-d_0-d_1)\Big). 
\end{align}

\begin{align}
D_{12}&= (\tau_1 -d_1)\tau_0 +\tau_1( \tau_0-d_1) +
2n_1(\tau_1 -d_1)(\tau_0- d_1) \nonumber \\
&=(n_1+1)\Big((\tau_1 -d_1)(\tau_0 -pd_1) 
+(\tau_1-pd_1)( \tau_0-d_1)\Big).
\end{align}

\begin{align}
D_{21}&= (\tau_1 -d_1 -d_2)(\tau_0 -d_0) +(\tau_1-d_2)( \tau_0-d_0-d_1) +
2n_1(\tau_1-d_1-d_2)(\tau_0-d_0-d_1)  \nonumber \\
&= (n_1+1)\Big((\tau_1 -d_1 -d_2)(\tau_0 -d_0-pd_1) 
+(\tau_1-pd_1-d_2)( \tau_0-d_0-d_1)\Big).
\end{align}

\begin{align}
D_{22}&= (\tau_1 -d_1 -d_2)\tau_0 +(\tau_1-d_2)( \tau_0 -d_1) +
2n_1(\tau_1-d_1-d_2)(\tau_0- d_1) \nonumber  \\
&=(n_1+1)\Big((\tau_1 -d_1 -d_2)(\tau_0-pd_1) 
+(\tau_1-pd_1-d_2)( \tau_0-d_1)\Big).
\end{align}

The tangent map $D_{F^{-1}x_0}F^{n_1+4}$ 
along the orbit segment $(F^{-1}x_0, x_0, x_1,\dots, F^{n_1} x_1, x_2, F x_2)$
(again, modulo the scalar $\frac{1}{d_1d_2}$) is
\begin{align}
G:= (d_1d_2)\cdot D_{F^{-1}x_0}F^{n_1+4}
=& \begin{bmatrix} 1 & 2 \\ 0 & 1 \end{bmatrix}
\cdot \begin{bmatrix} D_{11} & D_{12} \\  D_{21} & D_{22} \end{bmatrix}
\cdot\begin{bmatrix} 1 & 2 \\ 0 & 1 \end{bmatrix}
=\begin{bmatrix} D_{11}+2D_{21} & D_{12}+2D_{22} \\  D_{21} & D_{22} \end{bmatrix}
\cdot\begin{bmatrix} 1 & 2 \\ 0 & 1 \end{bmatrix}  \nonumber \\
=&\begin{bmatrix} 
D_{11}+2D_{21} & 2D_{11}+4D_{21} + D_{12}+2D_{22} \\  
D_{21} & 2D_{21}+ D_{22} 
\end{bmatrix},
\end{align}
where the four entries of the matrix $\hat G:= \frac{1}{n_1+1} G$ (modulo a common factor $n_1+1$) are
\begin{align}
\hat G_{11}&:=\frac{1}{n_1+1}(D_{11}  + 2D_{21})  \nonumber \\
&=3(\tau_1 -d_1 -\frac{2}{3}d_2)(\tau_0 -d_0-pd_1) + 3(\tau_1 -pd_1 -\frac{2}{3}d_2)( \tau_0-d_0-d_1),
\label{hatG11}
\end{align}

\begin{align}
\hat G_{12}&:=\frac{1}{n_1+1}(2D_{11}  +4D_{21} +  D_{12}+2D_{22}) \nonumber  \\
&=9(\tau_1 -d_1 -\frac{2}{3}d_2)(\tau_0 -\frac{2}{3}d_0-pd_1) + 9(\tau_1-pd_1-\frac{2}{3}d_2)(\tau_0-\frac{2}{3}d_0-d_1),
\label{hatG12}
\end{align}

\begin{align}
\hat G_{21}&:=\frac{1}{n_1+1}D_{21}  \nonumber \\
&= (\tau_1 -d_1 -d_2)(\tau_0 -d_0-pd_1) 
+(\tau_1-pd_1-d_2)( \tau_0-d_0-d_1),
\label{hatG21}
\end{align}

\begin{align}
\hat G_{22}&:=\frac{1}{n_1+1}(D_{22}  + 2D_{21})  \nonumber \\
&=3(\tau_1 -d_1 -d_2)(\tau_0 -\frac{2}{3}d_0 - pd_1) +3(\tau_1 -pd_1 -d_2)(\tau_0 -\frac{2}{3}d_0 -d_1).
\label{hatG22}
\end{align}

We have the following observation:
\begin{pro}\label{multi} 
Let $p=\frac{n_1}{n_1 +1}$.
If $\tau_0< \frac{2}{3}d_0+ pd_1$ and $\tau_1< pd_1 +\frac{2}{3}d_2$,
then the orbit segment $(F^{-1} x_0, \dots, F x_2)$ has positive derivative.
\end{pro}
\begin{proof}
If $\tau_0< \frac{2}{3}d_0+ pd_1$ and $\tau_1< pd_1 +\frac{2}{3}d_2$,
then all four entries of the matrix $\hat G$ are positive.
Therefore, all four entries of $DF^{n_1 +4}$ along the orbit segment $(F^{-1} x_0, \dots, F x_2)$
are positive. It follows that the orbit segment $(F^{-1} x_0, \dots, F x_2)$ has positive derivative.
\end{proof}

\begin{pro}\label{n1=2good}
Let  $R\ge 33.2r$.
If $n_1\ge 2$, then
$\tau_0< \frac{2}{3}d_0+ \frac{2}{3}d_1$ and $\tau_1< \frac{2}{3}d_1 +\frac{2}{3}d_2$.
\end{pro}
\begin{proof}
Let $n_1\ge 2$ be given. 
We will give some preliminary estimates first.
There are exactly $n_1$ complete chords on $\Gamma_R$.
It follows from \eqref{estitheta2n1} that for $i=0, 2$
\begin{align}
|d_i-  r\sin\phi_{\ast}| = |r\sin\theta_i - r\sin\phi_{\ast}| < \frac{17 r^2\sin\phi_{\ast}}{4R}.
\label{d0d2n}
\end{align}
Using \eqref{estiphi2n1} and a similar argument as in \eqref{tau0tau1}, we get: 
\begin{align}
 2r\sin\phi_{\ast}< \tau_0 + 2n_1d_1 + \tau_1 <
 |\Gamma_R| + r\cdot |\phi_0+\phi_{\ast}| + r\cdot |\phi_2 - \phi_{\ast}|
 < 2 r\sin\phi_{\ast}+ \frac{8.1 r^2\sin\phi_{\ast}}{R},
 \label{eq.tau01n1}
\end{align}
for $R\ge 33r$. Since $\tau_i < 2d_1$, $i=0, 1$, 
we have
\begin{align}
2r\sin\phi_{\ast} < \tau_0 + 2n_1d_1 +\tau_1 < (2n_1+4) d_1, \label{d1n1}
\end{align}
and hence $d_1> \frac{r\sin\phi_{\ast}}{n_1+2}$.

Now we are ready to prove the two items in the proposition.
Set $\alpha = \frac{\tau_0}{d_1}$. Note that $\alpha\in(0,2)$.
Then
\begin{align}
2d_0 >2 r\sin\phi_{\ast} - \frac{17 r^2\sin\phi_{\ast}}{2R}  
>\tau_0 + 2n_1d_1 + \tau_1 -  \frac{16.6 r^2\sin\phi_{\ast}}{R}  
> \alpha d_1 + 2n_1 d_1 + 0 - \frac{16.6 r^2\sin\phi_{\ast}}{R}.
\label{tau0alpha}
\end{align}
It follows that 
\begin{align}
 \frac{2}{3}d_0+  \frac{2}{3}d_1 - \tau_0
 & >  \frac{1}{3} \Big(\alpha d_1+  2n_1d_1 - \frac{19r^2\sin\phi_{\ast}}{R}\Big) + \frac{2}{3}d_1- \alpha d_1  \nonumber \\
& = \frac{2n_1 + 2- 2\alpha}{3}d_1  - \frac{16.6 r^2\sin\phi_{\ast}}{3R}
>  \frac{2n_1-2}{3(n_1+2)}r\sin\phi_{\ast} - \frac{16.6 r^2\sin\phi_{\ast}}{3R}  \nonumber \\
&\ge \frac{1}{6}r\sin\phi_{\ast} - \frac{16.6 r^2\sin\phi_{\ast}}{3R},
\end{align}
since $\frac{2n_1-2}{3(n_1+2)} \ge \frac{1}{6}$ for $n_1\ge 2$.
Therefore, 
$\frac{2}{3}d_0+  \frac{2}{3}d_1 - \tau_0>0$ when $\frac{1}{6}r \ge \frac{16.6 r^2 }{3R}$.
The later holds for $R\ge  33.2 r$. 

In the same way we get $ \frac{2}{3}d_1 +\frac{2}{3}d_2 -\tau_1>0$ for $R\ge  33.2 r$.
This completes the proof.
\end{proof} 

Note that  $p=\frac{n_1}{n_1+1}\ge \frac{2}{3}$ for $n_1 \ge 2$.
Combining Proposition \ref{multi} with Proposition \ref{n1=2good},
we get
\begin{pro}\label{n1=2defocus}
Let $R\ge 33.2 r$ be fixed. Then for each $x\in \hM$, if $n_1\ge 2$,
then the orbit segment $(F^{-1} x_0, \dots, F x_2)$  has positive derivative.
\end{pro} 

For $n_1=1$, we have $p=\frac{n_1}{n_1+1}=\frac{1}{2}$,
and the four entries of the matrix $G$ given in
 \eqref{hatG11},  \eqref{hatG12},
 \eqref{hatG21} and  \eqref{hatG22} are 

\begin{align}
G_{11}& =6(\tau_1 -d_1 -\frac{2}{3}d_2)(\tau_0 -d_0-\frac{1}{2}d_1) 
+ 6(\tau_1-\frac{1}{2}d_1-\frac{2}{3}d_2)( \tau_0-d_0-d_1),\\
G_{12}& =18(\tau_1 -d_1 -\frac{2}{3}d_2)(\tau_0 -\frac{2}{3}d_0-\frac{1}{2}d_1) 
+ 18(\tau_1 -\frac{1}{2}d_1 -\frac{2}{3}d_2)( \tau_0 -\frac{2}{3}d_0 - d_1),\\
G_{21}& = 2(\tau_1 -d_1 -d_2)(\tau_0 -d_0-\frac{1}{2}d_1) +2(\tau_1-\frac{1}{2}d_1-d_2)( \tau_0-d_0-d_1),\\
G_{22}& =6(\tau_1 -d_1 -d_2)(\tau_0 -\frac{2}{3}d_0 - \frac{1}{2}d_1)
 +6(\tau_1 -\frac{1}{2}d_1 -d_2)( \tau_0 -\frac{2}{3}d_0 -d_1).
\end{align}
Then Proposition \ref{multi}  implies that the orbit segment $(F^{-1} x_0, \dots, F x_2)$ has positive derivative
if $\tau_0< \frac{2}{3}d_0+ \frac{1}{2}d_1$ and $\tau_1<  \frac{1}{2}d_1 +\frac{2}{3}d_2$.
In the remaining two subsections,
we will show that the orbit segment $(F^{-1} x_0, \dots, F x_2)$ has positive derivative
for the remaining cases.
Note that there do exist orbits such that either (1) $\tau_0\ge  \frac{2}{3}d_0+ \frac{1}{2}d_1$ 
or (2) $\tau_1\ge  \frac{1}{2}d_1 +\frac{2}{3}d_2$. See Fig.~\ref{fig.taularge}.
It follows from Eq.~\eqref{eq.tau01n1} that these two inequalities cannot hold simultaneously.

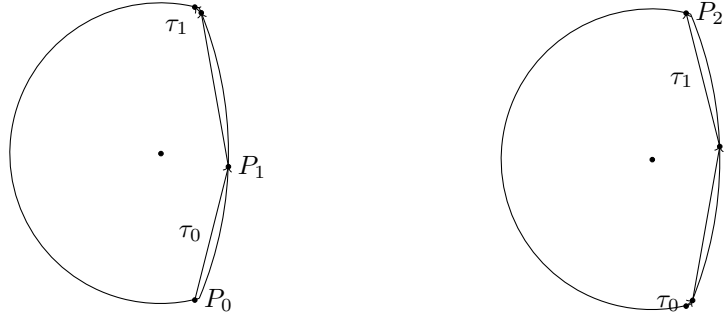
\begin{figure}[h]\tikzmath{
\a=75;
\b=22.73; 
}
\centering
\begin{tikzpicture}
\coordinate (OR) at (0,0);
\coordinate (Or) at (4.1,0);
\filldraw (4.1,0) circle(.03);
\coordinate (A) at ({5*cos(22.8)},{5*sin(22.8)});
\draw (A) arc (\a:(360 - \a):2);
\draw (A) arc (\b: -\b :5);
\coordinate (P) at ({4.1+2*cos(-77)},{2*sin(-77)});
\filldraw (P) circle(.03) node[right] {$P_0$} ;
\coordinate (P1) at ({5*cos(-2)},{5*sin(-2)});
\filldraw (P1) circle(.03) node[right] {$P_1$};
\coordinate (P2) at ({5*cos(22)},{5*sin(22)});
\filldraw (P2) circle(.03);
\coordinate (P3) at  ({4.1+2*cos(77)},{2*sin(77)});
\filldraw (P3) circle(.03);
\draw[->] (P) -- (P1) node[pos=0.5, left]{$\tau_0$} ;
\draw[->] (P1) -- (P2);
\draw[->] (P2) -- (P3)  node[pos=0.5,below left]{$\tau_1$};
\end{tikzpicture}
\hspace{1in}
\begin{tikzpicture}
\coordinate (OR) at (0,0);
\coordinate (Or) at (4.1,0);
\filldraw (4.1,0) circle(.03);
\coordinate (A) at ({5*cos(22.8)},{5*sin(22.8)});
\draw (A) arc (\a:(360 - \a):2);
\draw (A) arc (\b: -\b :5);
\coordinate (P) at ({4.1+2*cos(-77)},{2*sin(-77)});
\filldraw (P) circle(.03);
\coordinate (P1) at ({5*cos(-22)},{5*sin(-22)});
\filldraw (P1) circle(.03);
\coordinate (P2) at ({5*cos(2)},{5*sin(2)});
\filldraw (P2) circle(.03);
\coordinate (P3) at  ({4.1+2*cos(77)},{2*sin(77)});
\filldraw (P3) circle(.03);
\draw[->] (P) -- (P1) node[pos=0.5, left]{$\tau_0$} ;
\draw[->] (P1) -- (P2);
\draw[->] (P2) -- (P3)  node[pos=0.5,left]{$\tau_1$}  node[right] {$P_2$};
\end{tikzpicture}
\caption{An illustration for  $\tau_0\ge  \frac{2}{3}d_0+ \frac{1}{2}d_1$ (left)
and  $\tau_1\ge  \frac{1}{2}d_1 +\frac{2}{3}d_2$ (right).}\label{fig.taularge}
\end{figure}

\subsection{The subcase when $n_1=1$ and  $\tau_0\ge  \frac{2}{3}d_0+ \frac{1}{2}d_1$}
\label{n1tau0}
In this subsection we show that the orbit segment $(F^{-1} x_0, x_0, x_1, Fx_1, x_2, F x_2)$
has positive derivative for points $x\in \hM$ with $\tau_0\ge  \frac{2}{3}d_0+ \frac{1}{2}d_1$. 
We start with some preliminary estimates:

1). Combining $\tau_0\ge  \frac{2}{3}d_0+ \frac{1}{2}d_1$ with $\tau_0< 2d_1$, we see that
\begin{align}
d_1> \frac {4}{9}d_0 >  \frac {4}{9}\cdot(r\sin\phi_{\ast} - \frac{17r^2\sin\phi_{\ast}}{4R}) 
= \frac {4}{9}r\sin\phi_{\ast} -  \frac{17r^2\sin\phi_{\ast}}{9R}>\frac{13}{30}r\sin\phi_{\ast}.
\label{d1lower}
\end{align}
The last inequality holds for $R> 170r$.

2). Since $\tau_0< 2d_1$, we have $2d_0+  \frac{16.6 r^2\sin\phi_{\ast}}{R} > \tau_0+ 2d_1 +\tau_1 > 2\tau_0$. 
Therefore,
\begin{align}
\tau_0<  d_0+  \frac{8.3 r^2\sin\phi_{\ast}}{R}.
\label{tau0d0r}
\end{align}

3). Combining $2\tau_0< \tau_0 + 2 d_1 < 2 r\sin\phi_{\ast} + \frac{8.1 r^2\sin\phi_{\ast}}{R}$
with the condition $\tau_0 \ge \frac{2}{3}d_0 + \frac{1}{2}d_1$, we get
\begin{align}
d_1< 2 r\sin\phi_{\ast} + \frac{8.1 r^2\sin\phi_{\ast}}{R}
- \frac{4}{3}d_0 < \frac{2}{3}r\sin\phi_{\ast} + \frac{13.8 r^2\sin\phi_{\ast}}{R} \le  0.7r\sin\phi_{\ast}.
\label{d1upper}
\end{align}
The last inequality holds for $R \ge  414 r$.

4). Since $2r\sin\phi_{\ast}<\tau_0+ 2d_1 +\tau_1 < 2r\sin\phi_{\ast}+\frac{8.1 r^2\sin\phi_{\ast}}{R}$
and $|d_i- r\sin\phi_{\ast}| < \frac{17r^2\sin\phi_{\ast}}{4R}$,  $i=0, 2$,
we have 
\begin{align}
\Big| \tau_1-d_2 - (d_0 - \tau_0 -2d_1)\Big| 
&<|d_0- r\sin\phi_{\ast}|+  |d_2- r\sin\phi_{\ast}| +   \frac{8.1 r^2\sin\phi_{\ast}}{R}
<  \frac{16.6 r^2\sin\phi_{\ast}}{R},
\label{tau1d2}\\
\Big| \tau_1-\frac{2}{3}d_2  -(\frac{4}{3}d_0 - \tau_0 -2d_1)\Big| 
&< \frac{4}{3}|d_0- r\sin\phi_{\ast}| +  \frac{2}{3} |d_2- r\sin\phi_{\ast}| + \frac{8.1 r^2\sin\phi_{\ast}}{R} 
<  \frac{16.6 r^2\sin\phi_{\ast}}{R}.
\label{tau1d2b}
\end{align}

5). Moreover, we note  that
\begin{enumerate}
\item[a).] since $\tau_0\ge  \frac{2}{3}d_0+ \frac{1}{2}d_1$, we have
$\tau_0 -d_0 - d_1>  - \frac{1}{3}d_0 - \frac{1}{2}d_1
> -   0.7 r\sin\phi_{\ast}$;

\item[b).] since $\tau_0 < d_0 + d_1$, we have
$\tau_0 -\frac{2}{3}d_0 -  \frac{1}{2} d_1 <\frac{1}{3}d_0 + \frac{1}{2}d_1 <  0.7 r\sin\phi_{\ast}$.
\end{enumerate}
Putting these two estimates together, we get that for $R\ge 414 r$,
\begin{align}
|\tau_0 -d_0 - d_1|  & <   0.7 r\sin\phi_{\ast}; \label{tdd}\\
|\tau_0 -d_0 -  \frac{1}{2}d_1| & <   0.7  r\sin\phi_{\ast}; \label{tdpd}\\
|\tau_0 - \frac{2}{3} d_0 -  d_1| & <   0.7 r\sin\phi_{\ast}; \label{tqdd}\\
|\tau_0 -\frac{2}{3}d_0 -  \frac{1}{2}d_1|  & <    0.7 r\sin\phi_{\ast}. \label{tqdpd}
\end{align}

Now we are ready to estimate the four entries of $G$. 
We will argue in the following order: the $(2,1)$-entry, 
the $(1,1)$-entry, the $(2,2)$-entry and the $(1,2)$-entry.

\noindent{\bf The $(2,1)$ entry}.
It follows from \eqref{d1lower} and \eqref{tau0d0r} that
 $0< \frac{3}{2}d_1 - \frac{1}{3}d_0 \le \tau_0 -d_0 + d_1 < d_1 +  \frac{8.3r^2\sin\phi_{\ast}}{R}$.
Combining with \eqref{tdd},  \eqref{tdpd} and  \eqref{tau1d2}, we have
\begin{align}
G_{21}& = 2(\tau_1 -d_1 -d_2)(\tau_0 -d_0-\frac{1}{2}d_1) +2(\tau_1-\frac{1}{2}d_1-d_2)( \tau_0-d_0-d_1)  \nonumber \\
& > 2(d_0 - \tau_0 -3d_1)(\tau_0 -d_0-\frac{1}{2}d_1) +2(d_0 - \tau_0-\frac{5}{2}d_1)( \tau_0-d_0-d_1)
-4  \cdot \frac{16.6r^2\sin\phi_{\ast}}{R}\cdot   0.7 r\sin\phi_{\ast}  \nonumber  \\
&= -4(\tau_0 -d_0 + d_1)^2 + 12 d_1^2 - \frac{46.5 r^3\sin^2\phi_{\ast}}{R}
>  -4(d_1 +  \frac{8.3r^2\sin\phi_{\ast}}{R})^2 + 12 d_1^2 - \frac{46.5 r^3\sin^2\phi_{\ast}}{R} \nonumber  \\
&> -4(2d_1 +  \frac{8.3 r^2\sin\phi_{\ast}}{R})\cdot \frac{8.3 r^2\sin\phi_{\ast}}{R} 
+ 8(\frac {13}{30}r\sin\phi_{\ast})^2  - \frac{46.5 r^3\sin^2\phi_{\ast}}{R}  \nonumber  \\
& > -4\cdot \frac{43}{30}\cdot \frac{8.3 r^3\sin^2\phi_{\ast}}{R} + \frac{338}{225}r^2 \sin^2\phi_{\ast}
 - \frac{46.5 r^3\sin^2\phi_{\ast}}{R} >0.
\end{align}
The last inequality holds for $R \ge 62.7 r$. We have assumed $R \ge 414 r$ 
when obtaining  \eqref{d1lower} and  \eqref{d1upper}.

\noindent{\bf The $(1,1)$ entry}.
It follows from  \eqref{d1lower} and   \eqref{tau0d0r} that
$0< - \frac{1}{2}d_0 + \frac{3}{2}d_1 \le \tau_0 - \frac{7}{6}d_0 + d_1< -\frac{1}{6}d_0  + d_1 +  \frac{8.3 r^2\sin\phi_{\ast}}{R}$.
Combining with \eqref{tdd},  \eqref{tdpd} and  \eqref{tau1d2b}, we have
\begin{align}
G_{11}& =6(\tau_1 -d_1 -\frac{2}{3}d_2)(\tau_0 -d_0-\frac{1}{2}d_1) 
+ 6(\tau_1-\frac{1}{2}d_1-\frac{2}{3}d_2)( \tau_0-d_0-d_1)  \nonumber  \\
& > 6(\frac{4}{3}d_0 - \tau_0 -3d_1)(\tau_0 -d_0-\frac{1}{2}d_1) 
+ 6(\frac{4}{3}d_0 - \tau_0 - \frac{5}{2}d_1)( \tau_0-d_0-d_1)
 -12\cdot \frac{16.6r^2\sin\phi_{\ast}}{R}\cdot  0.7r\sin\phi_{\ast} \nonumber \\
&> -12 (\tau_0 - \frac{7}{6}d_0 + d_1)^2 + \frac{1}{3} d_0^2 -  7d_0 d_1+36 d_1^2
 - \frac{139.5 r^3\sin^2\phi_{\ast}}{R}  \nonumber \\
& > -12 ( -\frac{1}{6}d_0  + d_1 +  \frac{8.3 r^2\sin\phi_{\ast}}{R})^2 
+ \frac{1}{3} d_0^2 - 7 d_0 d_1+36 d_1^2
 - \frac{139.5 r^3\sin^2\phi_{\ast}}{R}  \nonumber  \\
 &= - 3 d_0 d_1+24 d_1^2  + (4d_0  - 24d_1)\frac{8.3 r^2\sin\phi_{\ast}}{R} 
 -\Big(\frac{8.3 r^2\sin\phi_{\ast}}{R}\Big)^2  - \frac{139.5 r^3\sin^2\phi_{\ast}}{R} \nonumber \\
 &> - 3 d_0 d_1+24 d_1^2  - \frac{107.4 r^3\sin^2\phi_{\ast}}{R} - \frac{0.1 r^3\sin^2\phi_{\ast}}{R}
 - \frac{139.5 r^3\sin^2\phi_{\ast}}{R} \\
&> (-\frac{27}{4} + 24) d_1^2 - \frac{247 r^3\sin^2\phi_{\ast}}{R}  
 > 3.2 r^2\sin^2\phi_{\ast} - \frac{247 r^3\sin^2\phi_{\ast}}{R}> 0,
\end{align}
where the second last inequality follows from Eq.~\eqref{d1lower} that $d_1 > \frac{4}{9}d_0$,
and the last inequality holds for $R \ge 77.2 r$. Again we have 
assumed $R\ge 414 r$ in obtaining  \eqref{d1upper}.

\noindent{\bf The $(2,2)$ entry}.
Combining with \eqref{tqdd},  \eqref{tqdpd} and  \eqref{tau1d2}, we have
\begin{align}
G_{22}& =6(\tau_1 -d_1 -d_2)(\tau_0 -\frac{2}{3}d_0 - \frac{1}{2}d_1)
 +6 (\tau_1 -\frac{1}{2}d_1 -d_2)( \tau_0 -\frac{2}{3}d_0 -d_1)  \nonumber  \\
& > 6(d_0 - \tau_0 -3d_1)(\tau_0 -\frac{2}{3}d_0 - \frac{1}{2}d_1)
+ 6(d_0 - \tau_0 - \frac{5}{2}d_1)( \tau_0 -\frac{2}{3}d_0 -d_1)
 -  \frac{139.5 r^3\sin^2\phi_{\ast}}{R}  \nonumber \\
&= -12 (\tau_0 - \frac{5}{6}d_0 + d_1)^2  + \frac{1}{3}d_0^2 - 7 d_0 d_1 + 36 d_1^2
 - \frac{139.5 r^3\sin^2\phi_{\ast}}{R} .
 \label{G22a}
 \end{align}

We will divide the estimate of this term into two subcases according to $\tau_0\le \alpha \cdot d_0$ or not.
To determine a proper value of $\alpha$, we need to 
consider the equation $E_{\alpha}(\lambda) =0$ of $\lambda$, where
\begin{align}
E_{\alpha}(\lambda)
:= & - 12\Big((\alpha-\frac{5}{6})\lambda+1\Big)^2 +\frac{1}{3}\lambda^2 -7\lambda + 36 \\
= & (-8+20\alpha - 12\alpha^2)\lambda^2  + (13- 24 \alpha)\lambda +24.
\end{align}
This function $E_{\alpha}(\lambda)$ appears later in \eqref{neweqG22}.
Note that $-8+20\alpha - 12\alpha^2 >0$ whenever $\alpha\in(\frac{2}{3},1)$. 
The two roots $\lambda_1(\alpha) \le \lambda_2(\alpha)$ of the equation $E_{\alpha}(\lambda) =0$ is
\begin{align}
\lambda_{1,2}(\alpha)
=\frac{-13 + 24 \alpha \pm \sqrt{(13- 24 \alpha)^2 - 4\cdot 24\cdot (-8+20\alpha - 12\alpha^2)}}{2(-8+20\alpha - 12\alpha^2)}.
\end{align}
Note that $\lambda_2(\alpha)>\lambda_1(\alpha)>2.25$ for all $0.7<\alpha<0.989814$.
Moreover, $\frac{2}{\alpha}<\frac{24}{11}$ for any $\alpha> \frac{11}{12} \simeq 0.91667$.
Then any choice of $\alpha  \in (0.91667, 0.989814)$ will work.

\noindent{\bf Case 1.} $\tau_0\le \alpha \cdot d_0$.
Then $0<\tau_0 -\frac{5}{6}d_0 +d_1\le (\alpha-\frac{5}{6})d_0 +d_1$. Continuing from \eqref{G22a}, we have 
 
\begin{align}
G_{22}& >  -12\Big((\alpha-\frac{5}{6})d_0 +d_1\Big)^2+\frac{1}{3}d_0^2-7d_0d_1 + 36d_1^2 
- \frac{139.5 r^3\sin^2\phi_{\ast}}{R}
\label{neweqG22} \\
&=d_1^2\cdot E_{\alpha}\big(\frac{d_0}{d_1}\big) - \frac{139.5 r^3\sin^2\phi_{\ast}}{R}
 \ge \big(\frac{13}{30}r\sin\phi_{\ast} \big)^2  \cdot E_{\alpha}(2.25) - \frac{139.5 r^3\sin^2\phi_{\ast}}{R},
\end{align}
since \eqref{d1lower} and \ref{d1upper} implies 
$\frac{d_0}{d_1} <2.25<\lambda_1(\alpha)$ and $d_1>\frac{13}{30} r\sin\phi_{\ast}$.
Then $G_{22}>0$ for $R\ge \big(\frac{30}{13}\big)^2\frac{139.5  r}{E_{\alpha}(2.25)}$.

\noindent{\bf Case 2.} $\tau_0> \alpha \cdot d_0$. Combining with
$\tau_0< 2d_1$, we have $d_0 <\frac{2}{\alpha} d_1$.
It follows from \eqref{tau0d0r} that
$0<\tau_0 -\frac{5}{6}d_0 +d_1< \frac{1}{6}d_0 +d_1+ \frac{19r^2\sin\phi_{\ast}}{2R}$.
Continuing from \eqref{G22a}, we have 
\begin{align}
G_{22} & > -12(\frac{1}{6}d_0 +d_1+ \frac{8.3 r^2\sin\phi_{\ast}}{R})^2
+\frac{1}{3}d_0^2-7d_0d_1 + 36d_1^2  - \frac{139.5 r^3\sin^2\phi_{\ast}}{R}  \nonumber  \\
&\ge  -\frac{1}{3}d_0^2 - 4d_0d_1 - 12 d_1^2  - \frac{192.5 r^3\sin^2\phi_{\ast}}{R} 
+\frac{1}{3}d_0^2-7d_0d_1 + 36d_1^2  - \frac{139.5 r^3\sin^2\phi_{\ast}}{R} \nonumber \\
&=-11d_0d_1 +24d_1^2   - \frac{332 r^3\sin^2\phi_{\ast}}{R}
= 11d_1^2\cdot \Big(\frac{24}{11} -\frac{d_0}{d_1}\Big)   - \frac{332 r^3\sin^2\phi_{\ast}}{R} \nonumber \\
&>11\cdot \big(\frac{13}{30}r\sin\phi_{\ast} \big)^2 \cdot \Big(\frac{24}{11} -\frac{2}{\alpha}\Big)   - \frac{332 r^3\sin^2\phi_{\ast}}{R},
\end{align}
since $d_0< \frac{31}{30}r\sin\phi_{\ast}$,
$\frac{13}{30}r\sin\phi_{\ast} < d_1<  0.7 r\sin\phi_{\ast} $ from \eqref{d1lower} and \eqref{d1upper}.
Then $G_{22}>0$ for $ R\ge \big(\frac{30}{13}\big)^2\frac{332 r}{(24 -\frac{22}{\alpha})}$.

For certainty, we pick $\alpha=0.9807$.
For this $\alpha$, we have $\lambda_2(\alpha)> \lambda_1(\alpha)> 2.25$,
and $\frac{2}{\alpha}< \frac{24}{11}$.
Then a sufficient condition for $G_{22}>0$ in both cases is $R\ge 1128.3 r$.

\noindent{\bf The $(1,2)$ entry}.
Combining with \eqref{tqdd},  \eqref{tqdpd} and  \eqref{tau1d2b}, we have
\begin{align}
G_{12}&=18(\tau_1 -d_1 -\frac{2}{3}d_2)(\tau_0 -\frac{2}{3}d_0-\frac{1}{2}d_1) 
+ 18(\tau_1 -\frac{1}{2}d_1 -\frac{2}{3}d_2)( \tau_0 -\frac{2}{3}d_0 - d_1)  \nonumber  \\
&>18(\frac{4}{3}d_0 - \tau_0 -3d_1)(\tau_0 -\frac{2}{3}d_0-\frac{1}{2}d_1) 
+ 18(\frac{4}{3}d_0 - \tau_0 -\frac{5}{2}d_1)( \tau_0 -\frac{2}{3}d_0 - d_1) 
 - \frac{418.5  r^3\sin^2\phi_{\ast}}{R}  \nonumber  \\
&=-36(\tau_0 -d_0 +d_1)^2 +4d_0^2 -42 d_0 d_1 + 108 d_1^2  - \frac{418.5  r^3\sin^2\phi_{\ast}}{R}.\label{G12a}
\end{align}

We divide the analysis of $G_{12}$ into two subcases according to $\tau_0\le \alpha \cdot d_0$ or not.
To determine a proper value of $\alpha$, we need to 
consider the  equation $F_{\alpha}(\lambda)=0$ of $\lambda$, where
\begin{align}
F_{\alpha}(\lambda)
&=-36((\alpha -1) \lambda +1)^2 +4\lambda^2 -42\lambda  + 108\\
&=(4-36(1-\alpha)^2) \lambda^2 + (30-72\alpha) \lambda + 72. 
\end{align}
This function $F_{\alpha}(\lambda)$ appears later in \eqref{neweqG12}.
Note that $4-36(1-\alpha)^2>0$ whenever $\alpha\in(\frac{2}{3},\frac{4}{3})$. 
Then the two roots $\lambda_1(\alpha) \le \lambda_2(\alpha)$ of the above equation are
\begin{align}
\lambda_{1,2}(\alpha)
&= \frac{-(30-72\alpha) \pm \sqrt{(30-72\alpha)^2 - 4\cdot 72\cdot(4-36(1-\alpha)^2)}}{2(4-36(1-\alpha)^2)}.
\end{align}
In the particular case that $\alpha=1$,  $F_{1}(\lambda) =4\lambda^2 -42\lambda  + 72$, which appears in
\eqref{F1}.
The two roots of the equation $F_{1}(\lambda)=0$ are
$\lambda_1(1)=\frac{42-6\sqrt{17}}{8}\simeq 2.15676$ and  
$\lambda_2(1)=\frac{42+6\sqrt{17}}{8}\simeq 8.34233$.
So $\frac{2}{\alpha}<\lambda_1(1)$ 
for any $\alpha>\frac{2}{\lambda_1(1)}\simeq 0.926925$.
Note that $\lambda_2(\alpha)>\lambda_1(\alpha)>2.25$ for all $0.7<\alpha<0.985887$.
We  pick $\alpha \in (0.926925, 0.985887)$.

\noindent{\bf Case 1.} $\tau_0\le \alpha d_0$.
Then  $0<\tau_0 -d_0 +d_1\le (\alpha -1) d_0 +d_1$.
Continuing from \eqref{G12a}, we have
\begin{align}
G_{12}&>-36((\alpha -1) d_0 +d_1)^2 +4d_0^2 -42 d_0 d_1 + 108 d_1^2 
 - \frac{418.5  r^3\sin^2\phi_{\ast}}{R}  \label{neweqG12}  \\
&= d_1^2\cdot F_{\alpha}\big(\frac{d_0}{d_1}\big) - \frac{418.5  r^3\sin^2\phi_{\ast}}{R} 
> \big(\frac{13}{30}r\sin\phi_{\ast} \big)^2 \cdot F_{\alpha}(2.25)  - \frac{418.5  r^3\sin^2\phi_{\ast}}{R},
\end{align}
since $\frac{d_0}{d_1} <2.25<\lambda_{1}(\alpha)<\lambda_2(\alpha)$
 and $d_1 > \frac{13}{30}r\sin\phi_{\ast}$ from \eqref{d1lower}.
Then $G_{12}>0$ for $R\ge \big(\frac{30}{13}\big)^2\frac{418.5  r}{F_{\alpha}(2.25)}$.

\noindent{\bf  Case 2.}  $\tau_0> \alpha \cdot d_0$. Combining with
$\tau_0< 2d_1$, we have $d_0 <\frac{2}{ \alpha} d_1$.
It follows from \eqref{tau0d0r} that
$0< \tau_0 - d_0 + d_1 <d_1+ \frac{8.3 r^2\sin\phi_{\ast}}{R}$.
Continuing from \eqref{G12a}, we have (for $R \ge 1000 r$)
\begin{align}
G_{12} &> -36(\frac{8.3 r^2\sin\phi_{\ast}}{R}+d_1)^2 +4d_0^2 -42 d_0 d_1 + 108 d_1^2  - \frac{418.5  r^3\sin^2\phi_{\ast}}{R}  \nonumber  \\
&\ge -36 d_1^2 - \frac{418.4 r^3\sin^2\phi_{\ast}}{R} - \frac{2480.1 r^4\sin^2\phi_{\ast}}{R^2} 
+ 4d_0^2   - 42 d_0 d_1 + 108 d_1^2   - \frac{418.5  r^3\sin^2\phi_{\ast}}{R} \label{F1}  \\
&= d_1^2\cdot F_{1}\big(\frac{d_0}{d_1}\big) - \frac{940 r^3\sin^2\phi_{\ast}}{R}
> \big(\frac{13}{30}r\sin\phi_{\ast} \big)^2 \cdot F_{1}\big(\frac{2}{ \alpha}\big) - \frac{940  r^3\sin^2\phi_{\ast}}{R},
\end{align}
since $\frac{d_0}{d_1}<\frac{2}{ \alpha}<\lambda_1(1) $ and
$\frac{13}{30}r\sin\phi_{\ast}< d_1 < 0.7r\sin\phi_{\ast}$ from \eqref{d1lower} and \ref{d1upper} .
Then $G_{12}>0$ for $R\ge  \big(\frac{30}{13}\big)^2\frac{ 940  r}{F_{1}(2/\alpha)}$.

For certainty, we pick $\alpha=0.9778$.
For this $\alpha$, we have   $\lambda_2(\alpha) >\lambda_1(\alpha) > 2.25$, 
$\frac{2}{\alpha} < \lambda_1(1)$.
Then a sufficient condition for $G_{12}>0$ in both cases is $R\ge 1773.7 r$.

Collecting terms, we see that all four entries of the matrix $G=(d_1d_2) DF^5$ along the orbit segment
$(F^{-1}x_0, x_0, x_1 F x_1, x_2, F x_2)$ are positive for  $R\ge 1773.7 r$. 
Therefore, the orbit segment has positive derivative for  $R\ge 1773.7 r$.
This complete the proof when $\tau_0\ge  \frac{2}{3}d_0+ \frac{1}{2}d_1$.

\subsection{The subcase when $n_1=1$ and $\tau_1\ge \frac{1}{2}d_1 + \frac{2}{3}d_2$}
\label{n1tau1}
There are two ways to deal with the case when $\tau_1\ge \frac{1}{2}d_1 + \frac{2}{3}d_2$:
\begin{enumerate}
\item either we run the same analysis as in \S \ref{n1tau0}  for the second time, 

\item or we use the time reversal property
of the billiard map.
\end{enumerate}
We will explain the second approach in details. 
Let $(F^{-1}x_0, x_0, x_1, F x_1, x_2, F x_2)$ be an orbit segment satisfying
 $\tau_1\ge \frac{1}{2}d_1 + \frac{2}{3}d_2$. 
Recall that the involution map $I:M\to M$ satisfies $F^{n}\circ I = I \circ F^{-n}$ for any $n\in\mathbb{Z}$.
Then the involution orbit of the above orbit segment, re-ordered in the positive direction,  is 
\begin{align}
(IFx_2, Ix_2, IFx_1,  Ix_1, x_0, IF ^{-1}x_0)
=(F^{-1}Ix_2, Ix_2, F^{-1}Ix_1,  Ix_1, x_0, F Ix_0).
\end{align} 
Note that this involution orbit satisfies the condition $\tau_0\ge  \frac{2}{3}d_0+ \frac{1}{2}d_1$, 
since 
\begin{align}
d(Ix)=\rho(\phi)\sin(\pi-\theta)=\rho(\phi)\sin(\theta)=d(x)
\end{align} 
for every $x\in M$. See Fig.~\ref{fig.taularge}, where the left figure can be viewed as the involution orbit
of the one on the right. 
Applying the result in Section \ref{n1tau0}, we see that
all four entries of the matrix $D_{IFx_2}F^5=\begin{bmatrix} a & b \\ c & d \end{bmatrix}$ are positive.
Taking derivatives of the time reversal symmetry equality, we get
\begin{align}
D_{F^{-1}x_0}F^5 &= D_{F^{-1}x_0}(I \circ F^{-5}\circ I)
=D_{F^{-1}Ix_2}I \circ (D_{IFx_2}F^{5})^{-1}\circ D_{F^{-1}x_0}I   \nonumber  \\
&=\frac{1}{ad-bc}\begin{bmatrix} 1 & 0 \\ 0 & -1 \end{bmatrix}
\begin{bmatrix} d & -b \\ -c & a \end{bmatrix}
\begin{bmatrix} 1 & 0 \\ 0 & -1 \end{bmatrix}
=\frac{1}{ad-bc} \begin{bmatrix} d & b \\ c & a \end{bmatrix}.
\end{align}
Note that $ad-bc=\det D_{IFx_2}F^5>0$ 
since $F$ preserves the area form $\omega$ on $M$.
It follows that all four entries of the matrix $D_{F^{-1}x_0}F^5$ are positive for  $R\ge 1773.7 r$.
Therefore, the orbit segment $(F^{-1}x_0, x_0, x_1, F x_1, x_2, F x_2)$ has positive derivative.
This finishes the proof of Proposition \ref{defocusing} (3).

Collecting the lower bounds on $R$, we see that the asymmetric lemon billiards $Q(\phi_{\ast},R)$ is hyperbolic 
for
\begin{align}\label{overall.bound}
R\ge \max \Big\{ \frac{16 r}{\delta_{\ast}\cdot \sin\phi_{\ast}}, \frac{165 r}{\sin^2\phi_{\ast}}, 1773.7 r \Big\}.
\end{align}
This completes the proof of Theorem \ref{mainthm}.

\begin{remark}
The bound on $R$ we got is not optimal, especially in the case when $\phi_{\ast}< \frac{\pi}{3}$.
More precisely, when $\phi_{\ast}< \frac{\pi}{4}$, one can define a slightly different first return set
$\hM' \subset M_{r}$ and show that the orbit segment $(F^{-2}x_0,\dots, F^2 x_2)$
is a subsegment of the orbit segment $(F^kx)_{0\le k\le \sigma'(x)}$, for every $x\in\hM'$ with $d_1< 2r$, 
where $\sigma'(x)$ is the first return time of a point $x\in \hM'$ to $\hM'$.
Then most of the estimates in Section \ref{n1=0} and Section \ref{n1ge1}
 can be significantly improved.
\end{remark}

\section*{Acknowledgements}
The authors are very grateful to the anonymous referee for useful comments and
suggestions, which helped them to improve the presentation of the paper.

\end{document}